\documentclass{article}
\usepackage{amsmath}
\usepackage{amsthm}
\usepackage{amssymb}
\usepackage{bm}
\usepackage[all]{xy}
\usepackage[dvips]{graphicx}
\usepackage{latexsym}
\usepackage{mathtools}
\usepackage{geometry}
 \renewcommand{\equation}

\newtheorem{prop}{Proposition}[section]
\newtheorem{lem}{Lemma}[section]
\newtheorem{thm}{Theorem}[section]

\newtheorem{fact}{Fact}[section]

\theoremstyle{definition}
\newtheorem{rmk}{Remark}[section]
\newtheorem{defini}{Definition}[section]

\title{The JSJ-decomposition of the 3-manifold obtained by $0$-surgery along a classical pretzel knot of genus one}
\author{Nozomu Sekino}
\date{}
\begin{document}
\maketitle

\begin{abstract}
We consider the JSJ-decomposition of the 3-manifold obtained by $0$-surgery along a classical pretzel knot of genus one. 
We use the classification of exceptional fillings of minimally twisted five-chain links by B. Martelli, C. Petronio and F. Roukema. 
\end{abstract}

\section{Introduction} \label{secintro}
For a knot in $S^{3}$ and a Seifert surface of minimal genus of it, there exists a taut finite depth foliation on the exterior of the knot such that the surface is a leaf (Theorem 3.1 of \cite{gabai}). 
Moreover, this foliation gives a foliation of the 3-manifold obtained by $0$-surgery along the knot by capping off the boundary of each leaf of the foliation (Corollary 8.2 of \cite{gabai}). 
On another hand, a problem of {\it characterizing slope} of a knot sometimes arises. 
This asks whether a Dehn surgery of given slope along a knot determines the knot or not. 
A rational number $r$ is called a {\it characterizing slope} of a knot $K$ if the 3-manifold obtained by $r$-surgery along a knot $J$ being homeomorphic to that of $K$ implies that $J$ is equivalent to $K$. 
There are many works toward this. 
For example, it is shown that every knot admits infinitely many characterizing slopes \cite{lackenby}. 
Concretely characterizing slopes and knots are constructed \cite{ozsvath} \cite{baldwin} \cite{baldwin2}, and so on. 
Also, non-characterizing slopes (i.e. slopes of inequivalent knots giving homeomorphic 3-manifolds) for some knots are also constructed \cite{brakes} \cite{osoinach} \cite{abe} \cite{yasui} \cite{miller} \cite{piccirillo} \cite{piccirillo2} \cite{manolescu}, and so on. 
For researches of characterizing slopes, it may be helpful to understand the structure of the 3-manifold obtained by a Dehn surgery along a knot. 
In this paper, together with a foliation stated above, we are interested in $0$-surgery (or longitudinal surgery) along a knot. 
Especially, we consider $0$-surgeries along classical pretzel knots (Definition~\ref{classical}) of genus one. 
There are two decomposing ways of 3-manifolds into simpler pieces. 
One is prime decomposition and the other is JSJ-decomposition (or torus decomposition) (Definition~\ref{jsj}). 
Note that a 3-manifold obtained by $0$-surgery along a knot is prime. 
About JSJ-decomposition, a 3-manifold obtained by $0$-surgery along a knot of genus one has an essential torus since the resultant of capping off a Seifert surface of minimal genus is essential \cite{scharlemann}. 
Thus decomposing tori of the JSJ-decomposition of a 3-manifold obtained by $0$-surgery along a knot of genus one is non-empty unless the entire manifold is a Seifert manifold. 
Moreover by using the result of \cite{ichihara}, we know that this entire manifold is a Seifert manifold if and only if the pretzel knot is a trefoil. 
Our result is the following.

\begin{thm}\label{main}
Let $\mathcal{T}$ be the set of decomposing tori of the JSJ-decomposition of the 3-manifold obtained by $0$-surgery along a pretzel knot $P_{2p+1,2q+1,2r+1}$. 
Then we have the following. 
\begin{itemize}
\item[(1.1.1)] If $\{-1, 0\} \subset \{p,q,r\}$, then $\mathcal{T}=\phi$ and the entire manifold is $S^{2}\times S^1$,  
\item[(1.1.2)] if $\{p,q,r\}=\{-1,\epsilon, \epsilon\}$ or $\{0,\epsilon-1,\epsilon-1\}$ for $\epsilon \in \{\pm1\}$, then $\mathcal{T}=\phi$ and the entire manifold is a closed torus bundle over $S^1$ whose monodromy is periodic of order six, 
\item[(1.1.3)] if $\{p,q,r\}=\{-1,\epsilon,k\}$ or $\{0,\epsilon-1,k-1\}$ for $\epsilon \in \{\pm1\}$ and some integer $k$ and not in the cases above, then $|\mathcal{T}|=1$ and the decomposed piece is a Seifert manifold whose base surface is an annulus with one exceptional point of order $|k|$, 
\item[(1.1.4)] if $\{p,q,r\}=\{-1,m,n\}$ or $\{0,m-1,n-1\}$ for some integers $m$ and $n$ and not in the cases above, then $|\mathcal{T}|=2$ and the decomposed pieces are two Seifert manifolds each of whose base surfaces are annuli with one exceptional points of order $|m|$ and $|n|$, 
\item[(1.1.5)] if $\{-2,1\}\subset \{p,q,r\}$, then $|\mathcal{T}|=1$ and the decomposed piece is a Seifert manifold whose base surface is an annulus with one exceptional point of order $2$, 
\item[(1.1.6)] if $\{p,q,r\}=\{-2,2,2\}$ or $\{-3,-3,1\}$, then $|\mathcal{T}|=2$ and the decomposed pieces are two Seifert manifolds, one of which is $\Sigma_{0,3}\times S^1$ and the other is the complement of the trefoil, 
\item[(1.1.7)] if $\{p,q,,r\}$ is in none of the above cases, then $|\mathcal{T}|=1$ and the decomposed piece is a hyperbolic manifold. 
\end{itemize}
\end{thm}

Some of the cases in Theorem~\ref{main} has already known. 
For example, the cases $(1.1.3)$ and $(1.1.4)$, where the corresponding pretzel knot is so called ``double twist knot'' are stated in \cite{sekino},  the case $(1.1.5)$ is stated in \cite{cantwell}. 
Our classification heavily depends on the result of \cite{martelli}, which classifies the exceptional fillings of the minimally twisted five-chain link. 

The rest of this paper is constructed as follows. 
In section~\ref{secdefinition}, we recall the definition of classical pretzel knots and introduce 3-manifolds related to them. 
We also give a surgery description of the 3-manifold obtained by cutting the resultant of $0$-surgery along a classical pretzel knot of genus one along the cap-offed standard Seifert surface of it. 
In section~\ref{secstructure}, we research the structure of the cut 3-manifold stated before. 
We list the cases where the manifolds are non-hyperbolic, and show that the manifolds are hyperbolic in the other cases by applying the classification of the exceptional fillings of the minimally twisted five-chain links. 
In section~\ref{secproof}, we prove Theorem~\ref{main} by using the arguments in section~\ref{secstructure}. 

\section{Classical pretzel knots whose genera are less than or equal to one and their $0$-surgeries} \label{secdefinition}
In this section, we introduce some links and 3-manifolds we need. 
Namely, classical pretzel links (or knots), the complement of a classical pretzel knot of genus one, the 3-manifold obtained by $0$-surgery along a classical pretzel knot of genus one and the 3-manifold obtained by cutting the $0$-surgered manifold along the torus which is obtained by capping off a Seifert surface of genus one of a classical pretzel knot of genus one. 
\subsection{Classical pretzel knots whose genera are less than or equal to one}

\begin{defini}\label{classical}(Classical pretzel links)\\
For three integers $l$, $m$ and $n$, the classical pretzel link $P_{l,m,n}$ is a link represented as in Figure~\ref{pretzel}, where ``R.H.T'' stands for ``right half twists''. 
We call the diagram in Figure~\ref{pretzel} the {\it standard diagram} of $P_{l,m,n}$. 
Note that $P_{l,m,n}$ is equivalent to $P_{m,n,l}$ and that $P_{l,m,n}$ is the mirror image of $P_{-l,-m,-n}$. 
\end{defini}

\begin{figure}[htbp]
 \begin{center}
  \includegraphics[width=100mm]{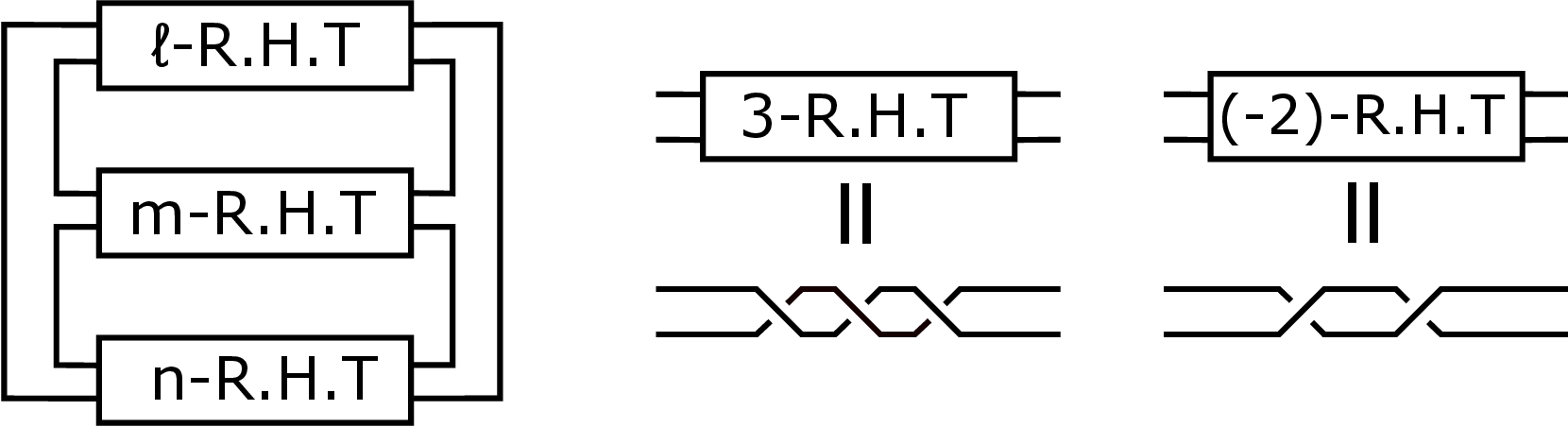}
 \end{center}
 \caption{the standard diagram of $P_{l,m,n}$ and examples of right half twists}
 \label{pretzel}
\end{figure}

For $P_{l,m,n}$ being a knot, there is at most one even number in $\{l,m,n\}$. 
At first, consider $P_{2k,m,n}$ whose genus is less than or equal to one. 
Using the fact that the Alexander polynomial of $P_{2k,m,n}$ is $(t+2+t^{-1})^{-1}\left( t^{m+n+1}+t^{m-n}+t^{-(m-n)}+t^{-(m+n+1)} +k\left( -t^{m+m+2}+t^{m+n}+t^{-(m+n)}-t^{-(m+n+2)} \right)\right)$ and focusing on the degrees, we see that $P_{2k,m,n}$ is a {\it twist knot} $K_{j}$ for some integer $j$, represented as in the left of Figure~\ref{twist} or the mirror image of $K_{j}$. 
Note that $P_{-1,-1,2j+1}$ is equivalent to $K_{j}$. 
Thus  when we consider a classical pretzel knot of genus one, we focus only on $P_{l,m,n}$ such that all of $\{l,m,n\}$ are odd numbers such as in Figure~\ref{pretzel_odd}, where ``R.F.T'' stands for ``right full twists''. 
In fact, by using one of the results in \cite{hirasawa}, which determines the genera of Montesinos knots, we see that a Montesinos knot of genus one is equivalent to $P_{l,m,n}$ for some odd integers $l$, $m$ and $n$. 

\begin{figure}[htbp]
 \begin{center}
  \includegraphics[width=120mm]{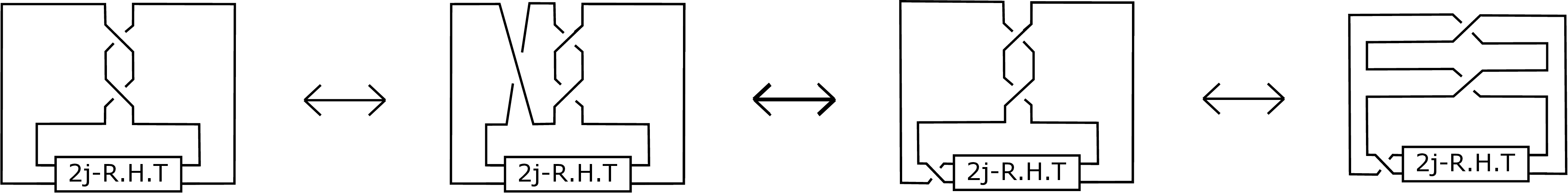}
 \end{center}
 \caption{the twist knot $K_{j}$ and $P_{-1,-1,2j+1}$}
 \label{twist}
\end{figure}

\begin{figure}[htbp]
 \begin{center}
  \includegraphics[width=100mm]{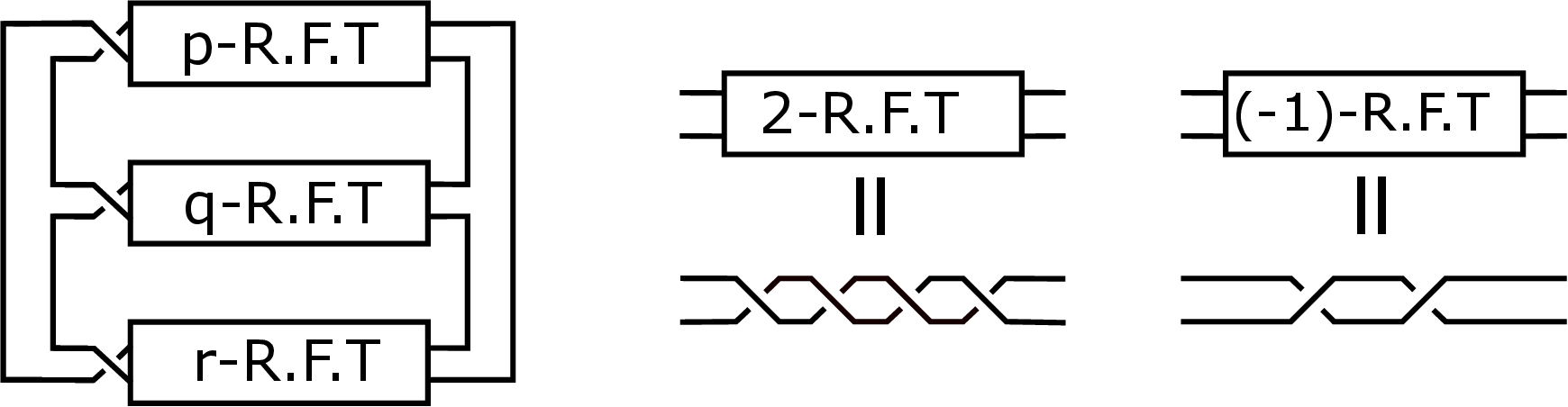}
 \end{center}
 \caption{$P_{2p+1,2q+1,2r+1}$ and examples of right full twists}
 \label{pretzel_odd}
\end{figure}

It can be known whether a given classical pretzel knot $P_{2p+1,2q+1,2r+1}$ is the unknot. 
We state this as a lemma though it is well-known.
\begin{lem} \label{unknot}
A knot $P_{2p+1,2q+1,2r+1}$ is the unknot if and only if $\{-1,0\}\subset \{p,q,r\}$.
\end{lem}
\begin{proof}
If part is clear. We will prove only if part. 
Suppose that $P_{2p+1,2q+1,2r+1}$ is the unknot. 
Since the Jones polynomial of $P_{2p+1,2q+1,2r+1}$ is $(t+2+t^{-1})^{-1}\biggl( (t+1+t^{-1})(-t^{-2(p+q+r+2)}-t^{-2(p+q+r+1)}+t^{-2(p+q+1)}+t^{-2(q+r+1)}+t^{-2(r+p+1)})+1\biggr)$, the identity $\biggl( (t+1+t^{-1})(-t^{-2(p+q+r+2)}-t^{-2(p+q+r+1)}+t^{-2(p+q+1)}+t^{-2(q+r+1)}+t^{-2(r+p+1)})+1\biggr)=t+2+t^{-1}$ must holds. 
Focus on the terms of even degree in the left hand side of this identity. 
It is clear that $-2(p+q+r+2)$ and $-2(p+q+r+1)$ are distinct integers. 
If none of $\{p,q,r\}$ is $-1$ nor $0$, then all of $\{ -2(p+q+1), -2(q+r+1), -2(r+p+1)\}$ are neither $-2(p+q+r+2)$ nor $-2(p+q+r+1)$. 
Thus there remain at least two terms of even degree in the left hand side of the identity after simplifying. This leads a contradiction and we see that $-1$ or $0$ must be in $\{p,q,r\}$. 
Moreover, since the Alexander polynomial of $P_{2p+1,2q+1,2r+1}$ is $\Bigl( (p+1)(q+1)(r+1)-pqr\Bigr)t-\biggl( 2\Bigl( (p+1)(q+1)(r+1)-pqr\Bigr)-1\biggr) + \Bigl( (p+1)(q+1)(r+1)-pqr\Bigr)t^{-1}$, the identity $(p+1)(q+1)(r+1)=pqr$ must hold. 
By this, we see that if one of $-1$ and $0$ is in $\{p,q,r\}$, then the other is also in. 
This finishes the proof.
\end{proof}

\subsection{The complement of a classical pretzel knot of genus one and the 3-manifold obtained by $0$-surgery along it}

\begin{defini}
A knot $P_{2p+1,2q+1,2r+1}$ for integers $p$, $q$ and $r$ bounds a torus $T$ with one boundary component in the standard diagram of $P_{2p+1,2q+1,2r+1}$ as in Figure~\ref{pretzel_odd_seifert}. 
We call this $T$ the {\it standard Seifert surface} of $P_{2p+1,2q+1,2r+1}$. 
Note that this $T$ is a Seifert surface of minimal genus unless $P_{2p+1,2q+1,2r+1}$ is the unknot and that this $T$ is compatible with taking mirror image and cyclic permutation of three parameters of a classical pretzel link. 
We call the light gray side (or the right side) of $T$ in Figure~\ref{pretzel_odd_seifert} the {\it front side} and the other side the {\it back side}.
\end{defini}

\begin{figure}[htbp]
 \begin{center}
  \includegraphics[width=100mm]{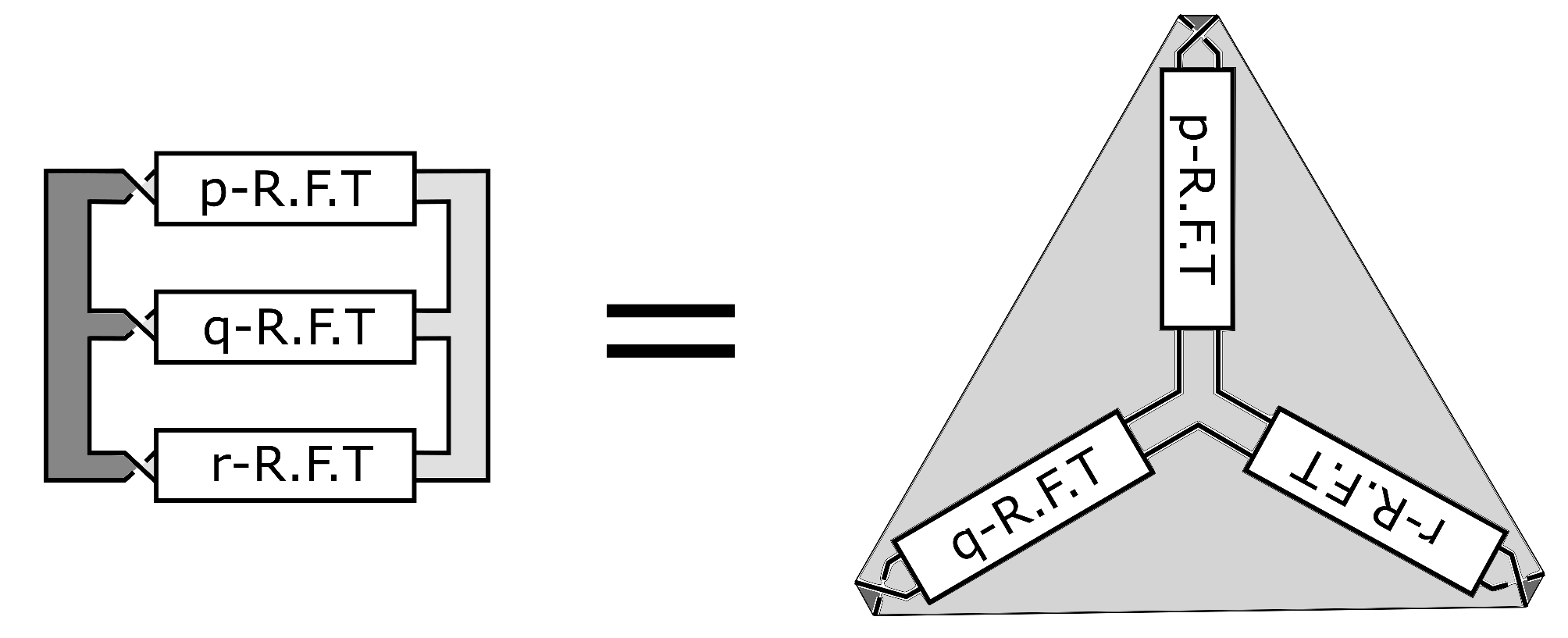}
 \end{center}
 \caption{The standard Seifert surface $T$ of genus one for $P_{2p+1,2q+1,2r+1}$}
 \label{pretzel_odd_seifert}
\end{figure}

Let $E_{2p+1,2q+1,2r+1}$ denote the closure of the complement of a small tubular neighborhood of $P_{2p+1,2q+1,2r+1}$ in $S^{3}$. 
We give $E_{2p+1,2q+1,2r+1}$ the orientation which comes from that of $S^{3}$. 
There is a useful presentation of the fundamental group of $E_{2p+1,2q+1,2r+1}$, so called the Lin presentation: 
\begin{align}
 \pi_{1}(E_{2p+1,2q+1,2r+1}) = \left<a, b, t \mid ta^{r+1}(ba)^{q}bt^{-1}=a^{r+1}(ba)^{q}, \ tb^{p+1}(ab)^{q}t^{-1}=b^{p+1}(ab)^{q}a \right> \label{eqlinpresentation}
\end{align}
In this presentation, $a$, $b$ and $t$ are based loops in $E_{2p+1,2q+1,2r+1}$ as in the left of Figure~\ref{fundamental}. 
This presentation is obtained as follows. 
The complement of the standard Seifert surface $T$ is a handlebody whose fundamental group is generated by loops $a$ and $b$. 
Take loops $x$ and $y$ as in the right of Figure~\ref{fundamental}. 
The push-ups of  loops $x$ and $y$ on $T$ is $a^{r+1}(ba)^{q}b$ and $b^{p+1}(ab)^{q}$, respectively and the push-downs of  loops $x$ and $y$ on $T$ is $a^{r+1}(ba)^{q}$ and $b^{p+1}(ab)^{q}a$, respectively. 
When pasting the cut ends of $T$ in the handlebody in order to reconstruct $E_{2p+1,2q+1,2r+1}$, the push-ups of $x$ and $y$ are identified with the push-downs of $x$ and $y$, respectively. 
In this presentation, the meridian of $P_{2p+1,2q+1,2r+1}$ corresponds to $t$ and the canonical longitude of it corresponds to $b^{p}(ba)^{q+1}a^{r}b^{-p}(b^{-1}a^{-1})^{q+1}a^{-r}$. 

\begin{figure}[htbp]
 \begin{center}
  \includegraphics[width=100mm]{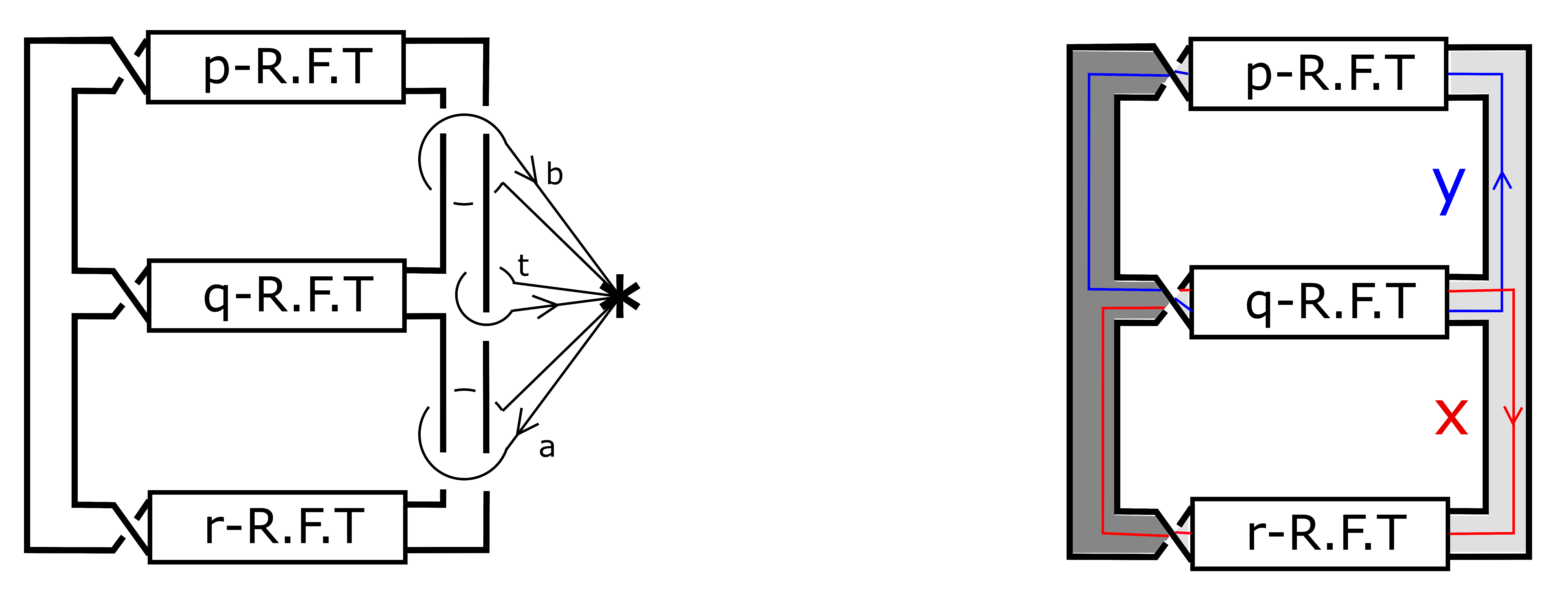}
 \end{center}
 \caption{Left: the based loops $a$, $b$ and $t$ in $E_{2p+1,2q+1,2r+1}$ \ \ \ \ Right: loops $x$ and $y$ on $T$}
 \label{fundamental}
\end{figure}

Let $P_{2p+1,2q+1,2r+1}(0)$ denote the 3-manifold obtained by $0$-surgery along $P_{2p+1,2q+1,2r+1}$. 
We give $P_{2p+1,2q+1,2r+1}(0)$ the orientation which comes from that of $E_{2p+1,2q+1,2r+1}$. 
In $P_{2p+1,2q+1,2r+1}(0)$, the boundary of the standard Seifert surface $T$ is capped-off by a disk. 
Let $T'$ denote this resultant closed torus. 
According to Theorem 10.1 of \cite{scharlemann}, this $P_{2p+1,2q+1,2r+1}(0)$ is irreducible and $T'$ is incompressible unless $P_{2p+1,2q+1,2r+1}$ is the unknot since $T$ is a Seifert surface of minimal genus. 
From the presentation \eqref{eqlinpresentation}, the fundamental group of $P_{2p+1,2q+1,2r+1}(0)$, denoted by $F_{p,q,r}$, can be computed as: 
\begin{align}\label{stdrep}
 F_{p,q,r} &= \pi_{1}\left(P_{2p+1,2q+1,2r+1}(0)\right) \notag \\
             &= \left<a, b, t \mid ta^{r+1}(ba)^{q}bt^{-1}=a^{r+1}(ba)^{q}, \ tb^{p+1}(ab)^{q}t^{-1}=b^{p+1}(ab)^{q}a, \ b^{p}(ba)^{q+1}a^{r}=a^{r}(ab)^{q+1}b^{p} \right>.
\end{align}

\subsection{The 3-manifold obtained by cutting $P_{2p+1,2q+1,2r+1}(0)$ along $T'$}
Moreover, we consider $\overline{P_{2p+1,2q+1,2r+1}(0)\setminus T'}$, the closure  of the complement of $T'$ in $P_{2p+1,2q+1,2r+1}(0)$. 
We give $\overline{P_{2p+1,2q+1,2r+1}(0)\setminus T'}$ the orientation which comes from that of $P_{2p+1,2q+1,2r+1}(0)$. 
This is obtained from the closure of the complement of the standard Seifert surface $T$ of $P_{2p+1,2q+1,2r+1}$ in $S^{3}$, which is a handlebody of genus two by attaching a $2$-handle along the longitude of $P_{2p+1,2q+1,2r+1}$. 
Following this construction, the fundamental group of $\overline{P_{2p+1,2q+1,2r+1}(0)\setminus T'}$, denoted by $\Gamma_{p,q,r}$, can be computed as: 
\begin{align}
 \Gamma_{p,q,r}= \pi_{1}\left( \overline{P_{2p+1,2q+1,2r+1}(0)\setminus T'}\right) = \left<a, b\mid b^{p}(ba)^{q+1}a^{r}=a^{r}(ab)^{q+1}b^{p} \right>. \label{gammapqr}
\end{align}
The boundary of $\overline{P_{2p+1,2q+1,2r+1}(0)\setminus T'}$ consists of two tori. 
Let $T'_{+}$ denote the one coming from the front side of $T'$ and $T'_{-}$ denote the other. 
We fix presentations of $\pi_{1}(T'_{\pm})$, denoted by $\Gamma_{\pm}$ as follows, where $X_{\pm}$ and $Y_{\pm}$ come from $x$ and $y$ in Figure~\ref{fundamental}. 
\begin{align}
\Gamma_{+}&=\pi_{1}(T_{+})=\left< X_{+}, Y_{+} \mid [X_{+}, Y_{+}]=1 \right>  \\ \notag \\
\Gamma_{-}&=\pi_{1}(T_{-})=\left< X_{-}, Y_{-} \mid [X_{-}, Y_{-}]=1 \right> 
\end{align}
Then we take homomorphisms $\iota _{+}: \Gamma_{+}\longrightarrow \Gamma_{p,q,r}$ and $\iota _{-}: \Gamma_{-}\longrightarrow \Gamma_{p,q,r}$ which are induced by the inclusions such that $\iota_{+}(X_{+})=a^{r+1}(ba)^{q}b$, $\iota_{+}(Y_{+})=b^{p+1}(ab)^{q}$, $\iota_{-}(X_{-})=a^{r+1}(ba)^{q}$ and $\iota_{-}(Y_{-})=b^{p+1}(ab)^{q}a$ hold. 
Note that  $\overline{P_{2p+1,2q+1,2r+1}(0)\setminus T'}$ is irreducible and has incompressible boundary unless $P_{2p+1,2q+1,2r+1}$ is the unknot since $P_{2p+1,2q+1,2r+1}(0)$ is irreducible and $T'$ is incompressible unless $P_{2p+1,2q+1,2r+1}$ is the unknot.

Next we will give a surgery description for $\overline{P_{2p+1,2q+1,2r+1}(0)\setminus T'}$. 
\begin{prop}\label{xpqr}
Let $X_{p,q,r}$ be a 3-manifold with a surgery description as in Figure~\ref{5chain_mirror}. 
We give $X_{p,q,r}$ the orientation which comes from that of $S^3$. 
Assume that $P_{2p+1,2q+1,2r+1}$ is not the unknot. 
Then $X_{p,q,r}$ is orientation preservely homeomorphic to $\overline{P_{2p+1,2q+1,2r+1}(0)\setminus T'}$. 
\end{prop}

\begin{figure}[htbp]
 \begin{center}
  \includegraphics[width=40mm]{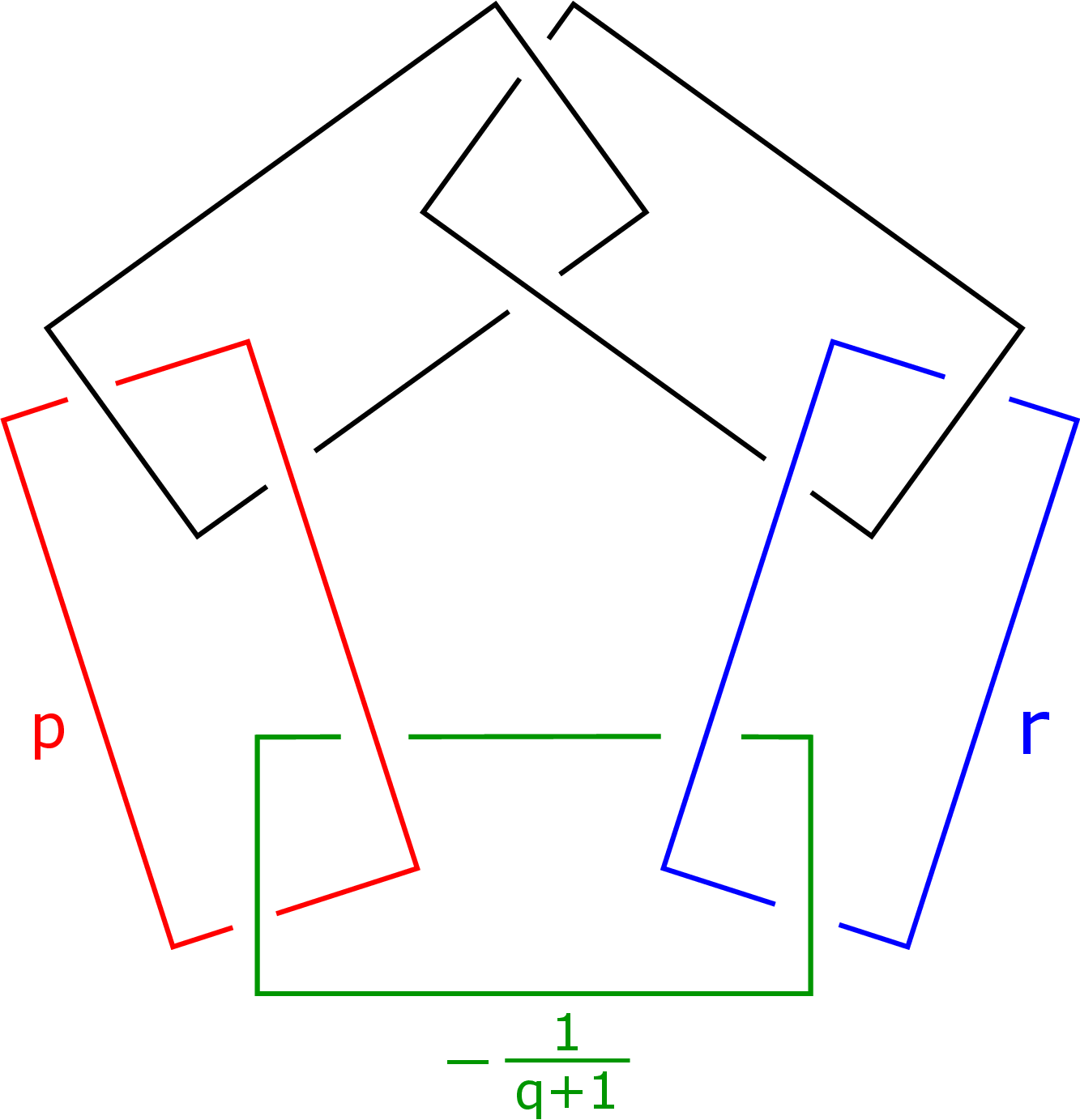}
 \end{center}
 \caption{A surgery description of $X_{p,q,r}$, where three of five boundary components of the closure of the complement of a small neighborhood of this five component link are filled. }
 \label{5chain_mirror}
\end{figure}

\begin{proof}
Take a five component link $L=C_{1}\cup C_{2}\cup C_{3} \cup C_{4} \cup C_{5}$ as in the top of Figure~\ref{5chain_mirror_wirtinger}. 
Let $E(L)$ denote the complement of a small neighborhood of $L$. 
In order to compute $\pi_{1}\left(E(L)\right)$, we assign an orientation to $L$ and symbols to arcs of the diagram of $L$ as in the bottom of Figure~\ref{5chain_mirror_wirtinger}. 
Note that the meridian and the canonical longitude of $C_1$ are represented as $A$ and $EB$, respectively, 
those of $C_2$ are represented as $B$ and $CA$, respectively, 
those of $C_3$ are represented as $C$ and $CDC^{-1}\cdot B$, respectively, 
those of $C_4$ are represented as $D$ and $EC$, respectively, 
and those of $C_5$ are represented as $E$ and $DA$, respectively. 
Then we compute  $\pi_{1}\left(E(L)\right)$ as follows. 
\begin{align}
\pi_{1}\left(E(L)\right) &=  \left<A, B, C, D, E \mid E^{-1}AE=BAB^{-1},\ \ ABA^{-1}=C^{-1}BC,\ \ BCB^{-1}=CD^{-1}CDC^{-1}\ \ \ \right> \notag \\
&\ \ \ \ \ \ \ \ \ \ \ \ \hspace{3.0cm},CDC^{-1}=E^{-1}DE,\ AEA^{-1}=D^{-1}ED  \notag \\
 &=\left<A, B, C, D, E \mid A(EB)=(EB)A,\ \ (CA)B=B(CA),\ \ DC^{-1}BC=CDC^{-1}B\ \ \ \right> \notag \\
&\ \ \ \ \ \ \ \ \ \ \ \ \hspace{3.0cm},CDC^{-1}=E^{-1}DE,\ AEA^{-1}=D^{-1}ED  \notag \\
 &=\left<A, B, C, D, E \mid A(EB)=(EB)A,\ \ (CA)B=B(CA),\ \ (CDC^{-1}B)C=C(CDC^{-1}B)\ \ \ \right> \notag \\
&\ \ \ \ \ \ \ \ \ \ \ \ \hspace{3.0cm},CDC^{-1}=E^{-1}DE,\ AEA^{-1}=D^{-1}ED 
\end{align}
$X_{p,q,r}$ is obtained by $p$-filling of $C_1$, $\left(-\frac{1}{q+1}\right)$-filling of $C_2$ and $r$-filling of $C_3$. 
Thus $\pi_{1}\left(X_{p,q,r}\right)$, denoted by $G_{p,q,r}$ is computed as follows.
\begin{align}
 G_{p,q,r}&=\pi_{1}\left( X_{p,q,r}\right) \notag \\
&=  \left<A, B, C, D, E \mid A(EB)=(EB)A,\ \ (CA)B=B(CA),\ \ (CDC^{-1}B)C=C(CDC^{-1}B)\ \ \ \right> \notag \\
&\hspace{3.0cm},CDC^{-1}=E^{-1}DE,\ AEA^{-1}=D^{-1}ED\notag \\
&\hspace{3.0cm},A^{p}=(EB)^{-1},\ B=(CA)^{q+1},\ C^{r}=(CDC^{-1}B)^{-1}  \notag \\
&=  \left<A, B, C, D, E \mid A^{p}=(EB)^{-1},\ B=(CA)^{q+1},\ C^{r}=(CDC^{-1}B)^{-1}\ \ \ \right> \notag \\
&\hspace{3.0cm},CDC^{-1}=E^{-1}DE,\ AEA^{-1}=D^{-1}ED\notag \\
&=  \left<A, B, C, D, E \mid E=A^{-p}B^{-1},\ B=(CA)^{q+1},\ D=C^{-(r+1)}B^{-1}C\ \ \ \right> \notag \\
&\hspace{3.0cm},CDC^{-1}=E^{-1}DE,\ AEA^{-1}=D^{-1}ED\notag \\
&=  \left<A, B, C, D, E \mid E=A^{-p}(CA)^{-(q+1)},\ B=(CA)^{q+1},\ D=C^{-r}(AC)^{-(q+1)}\ \ \ \right> \notag \\
&\hspace{3.0cm},CDC^{-1}=E^{-1}DE,\ AEA^{-1}=D^{-1}ED\notag \\
&=  \left<A, B, C, D, E \mid E=A^{-p}(CA)^{-(q+1)},\ B=(CA)^{q+1},\ D=C^{-r}(AC)^{-(q+1)}\ \ \ \right> \notag \\
&\hspace{3.0cm},C\cdot C^{-r}(AC)^{-(q+1)}\cdot C^{-1}=(CA)^{q+1}A^{p}\cdot C^{-r}(AC)^{-(q+1)}\cdot A^{-p}(CA)^{-(q+1)} \notag \\
&\hspace{3.0cm},\ A\cdot A^{-p}(CA)^{-(q+1)} \cdot A^{-1}=(AC)^{q+1}C^{r}\cdot A^{-p}(CA)^{-(q+1)}\cdot C^{-r}(AC)^{-(q+1)}\notag \\
&=  \left<A, B, C, D, E \mid E=A^{-p}(CA)^{-(q+1)},\ B=(CA)^{q+1},\ D=C^{-r}(AC)^{-(q+1)}\ \ \ \right> \notag \\
&\hspace{3.0cm}, C^{-r}=(CA)^{q+1}A^{p}\cdot C^{-r}(AC)^{-(q+1)}\cdot A^{-p} \notag \\
&\hspace{3.0cm},\  A^{-p}=(AC)^{q+1}C^{r}\cdot A^{-p}(CA)^{-(q+1)}\cdot C^{-r}\notag \\
&=  \left<A, C \mid A^{p}(AC)^{q+1}C^{r}=C^{r}(CA)^{q+1}A^{p}\right> \label{gpqr}.
\end{align}
Comparing (\ref{gammapqr}) and (\ref{gpqr}), we see that $\Gamma_{p,q,r}=\pi_{1}\left( \overline{P_{2p+1,2q+1,2r+1}(0)\setminus T'}\right)$ and $G_{p,q,r}=\pi_{1}(X_{p,q,r})$ are isomorphic. 
Recall that $\overline{P_{2p+1,2q+1,2r+1}(0)\setminus T'}$ is irreducible and has incompressible boundary since $P_{2p+1,2q+1,2r+1}$ is not the unknot. 
We will show that $X_{p,q,r}$ is also irreducible. 
Suppose that $X_{p,q,r}$ is reducible for a contradiction. 
By noting that $X_{p,q,r}$ has boundaries which are not spherical and using Poincar\'{e} conjecture, the group $G_{p,q,r}=\pi_{1}(X_{p,q,r})$ admits a structure of non-trivial free product, and so does $\Gamma_{p,q,r}=\pi_{1}\left( \overline{P_{2p+1,2q+1,2r+1}(0)\setminus T'}\right)$. 
Then $\overline{P_{2p+1,2q+1,2r+1}(0)\setminus T'}$ would be non-prime by Kneser conjecture (see Theorem 2.1 of \cite{aschenbrenner} for example). 
This leads a contradiction and we see that $X_{p,q,r}$ is irreducible. 

Let $\partial _{+}$ and $\partial _{-}$ be boundary tori of $X_{p,q,r}$ which come from $C_{4}$ and $C_{5}$, respectively. 
We fix presentations of the fundamental groups $G_{\pm}$ of $\partial_{\pm}$ as follows. 
\begin{align}
G_{+}&=\pi_{1}(\partial_{+})=\left< x_{+}, y_{+} \mid [x_{+}, y_{+}]=1 \right>   \\
\notag \\
G_{-}&=\pi_{1}(\partial_-)=\left< x_{-}, y_{-} \mid [x_{-}, y_{-}]=1 \right>. 
\end{align}  
We choose $x_{\pm}$ and $y_{\pm}$ above so that $j_{+}(x_{+})=(AC)^{q}AC^{r+1}(=D^{-1})$, $j_{+}(y_{+})=(AC)^{q}A^{p+1}\left(=(EC)^{-1} \right)$, $j_{-}(x_{-})=C(AC)^{q}C^{r}\left(=(DA)^{-1} \right)$ and $j_{-}(y_{-})=C(AC)^{q}A^{p+1}(=E^{-1})$ hold  for some homomorphisms $j_{\pm}:G_{\pm}\longrightarrow G$ induced by the inclusions. 

Set $\phi$, $\phi_{+}$, $\phi_{-}$, $f$ and $g$ to be isomorphisms satisfying: 
\begin{itemize}
\item $\phi:G_{p,q,r}\longrightarrow \Gamma_{p,q,r}$ maps $A$ and $C$ to $b$ and $a$ respectively,
\item $\phi_{+}:G_{+}\longrightarrow \Gamma_{+}$ maps $x_{+}$ and $y_{+}$ to $X_{+}$ and $Y_{+}$ respectively, 
\item $\phi_{-}:G_{-}\longrightarrow \Gamma_{-}$ maps $x_{-}$ and $y_{-}$ to $X_{-}$ and $Y_{-}$ respectively, 
\item $f:\Gamma \longrightarrow \Gamma$ maps $z\in \Gamma$ to $(ba)^{q}b\cdot z\cdot b^{-1}(ba)^{-q}$, and 
\item $g:\Gamma \longrightarrow \Gamma$ maps $z\in \Gamma$ to $a(ba)^{q}\cdot z\cdot (ba)^{-q}a^{-1}$. 
\end{itemize}
Then the following two diagrams commute.
\[
 \xymatrix{
   \Gamma_{+} \ar[r]^{\iota_{+}} & \Gamma \ar[r]^{f}  & \Gamma &  & \Gamma_{-} \ar[r]^{\iota_{-}} & \Gamma \ar[r]^{g}  & \Gamma\\
   G_{+}\ar[u]^{\phi_{+}} \ar[r]^{j_{+}} &  G\ar[ur]_{\phi} & & &  G_{-}\ar[u]^{\phi_{-}} \ar[r]^{j_{-}} &  G\ar[ur]_{\phi} 
}
\]
Therefore, $\phi$ (with other isomorphisms) is an isomorphism preserving the peripheral structure. 
Note that $\overline{P_{2p+1,2q+1,2r+1}(0)\setminus T'}$ and $X_{p,q,r}$ are irreducible and that they have essential surfaces since a 3-manifold whose boundary component other than spheres are not empty has an essential surface. 
By a result of Waldhausen \cite{waldhausen}, this isomorphism is induced by a homeomorphism between $\overline{P_{2p+1,2q+1,2r+1}(0)\setminus T'}$ and $X_{p,q,r}$. 

Moreover, the oriented basis of loops on the boundary of small neighborhood of $C_4$ and $C_5$ are mapped to those of $T'_{+}$ and $T'_{-}$, which are compatible to our orientations of $\overline{P_{2p+1,2q+1,2r+1}(0)\setminus T'}$ and $X_{p,q,r}$. 
Thus the homeomorphism stated above preserves the orientation. 

\begin{figure}[htbp]
 \begin{center}
  \includegraphics[width=120mm]{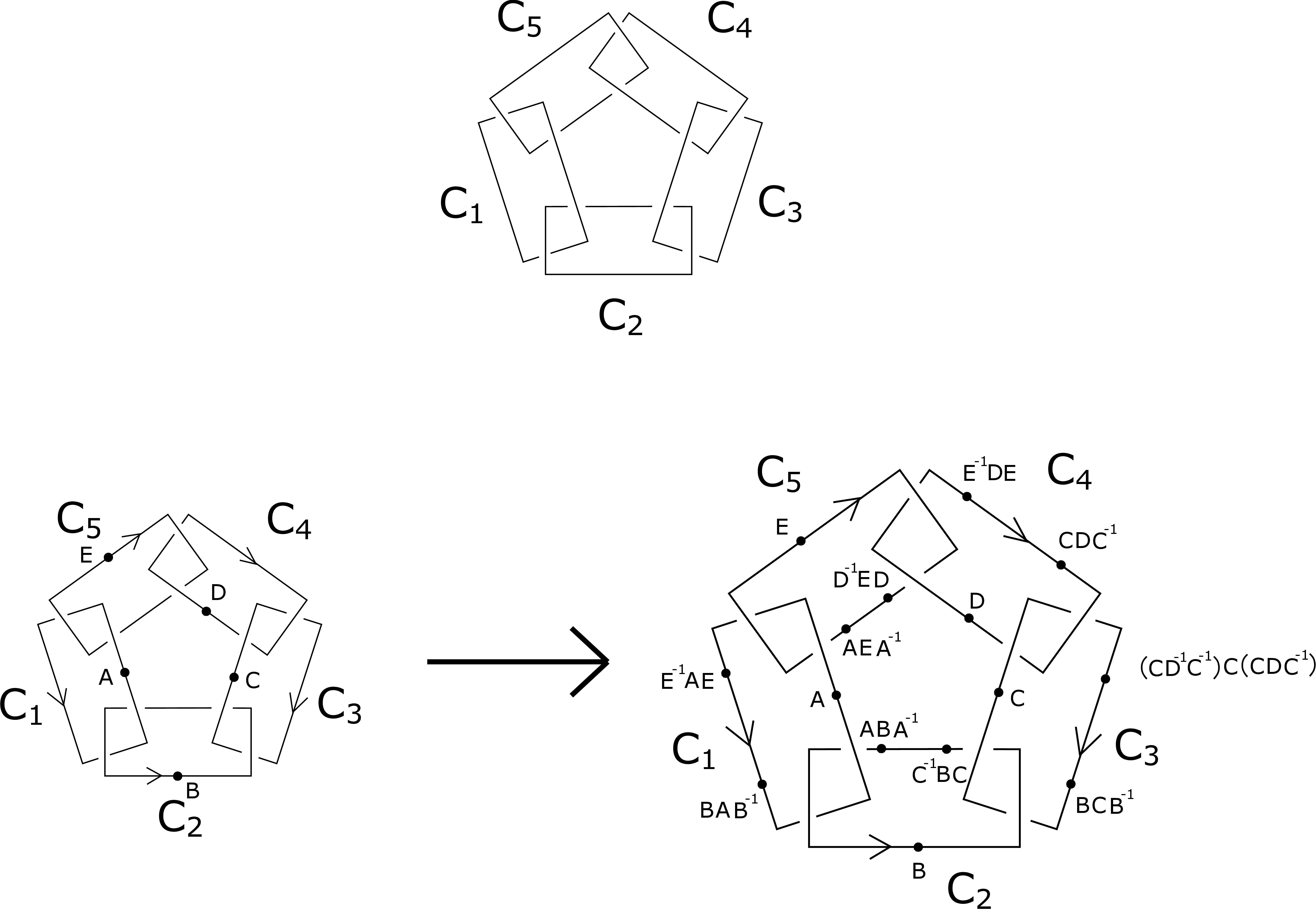}
 \end{center}
 \caption{Top: five component link  \ \ \ \ \ \ Bottom: assignments for a Wirtinger presentation of $\pi_{1}\left(E(L)\right)$}
 \label{5chain_mirror_wirtinger}
\end{figure}

\end{proof}

\section{The structure of $\overline{P_{2p+1,2q+1,2r+1}(0)\setminus T'}$} \label{secstructure}
In this section we prove the following proposition, which determines whether the resultant of cutting $P_{2p+1,2q+1,3r+1}(0)$ along $T'$ is hyperbolic or not: 

\begin{prop} \label{proposition}
Let $P_{2p+1,2q+1,2r+1}$ be a classical pretzel knot other than the unknot. 
Then $\overline{P_{2p+1,2q+1,2r+1}(0)\setminus T'}$ is non-hyperbolic if and only if either of the following holds: 

\begin{itemize}
\item[(3.1.1)] $-1\in \{p,q,r\}$, 
\item[(3.1.2)] $0\in \{p,q,r\}$,
\item[(3.1.3)] $\{-2,1\}\subset \{p,q,r\}$,
\item[(3.1.4)] $\{p,q,r\}=\{-2,2,2\}$ or
\item[(3.1.5)] $\{p,q,r\}=\{-3,-3,1\}$.
\end{itemize}
\end{prop}

\subsection{Non-hyperbolic cases}\label{nonhyp}
Through the identification in Proposition~\ref{xpqr} and classification result in \cite{martelli}, we can find the structures of $\overline{P_{2p+1,2q+1,2r+1}(0)\setminus T'}$ which are non-hyperbolic in \cite{martelli}. 
However, we list the structures for completeness. 

\subsubsection{The case where $-1\in \{p,q,r\}$ }
Since $P_{2p+1,2q+1,2r+1}$ is equivalent to $P_{2q+1,2r+1,2p+1}$ and the standard Seifert surface $T$ is compatible with this equivalence, we assume that $q=-1$. 
We use Proposition~\ref{xpqr} and the surgery description in Figure~\ref{5chain_mirror}. 
Then we change the diagram as in Figure~\ref{5-chain_mirror_qminus1}. 
By the torus depicted in the right of Figure~\ref{5-chain_mirror_qminus1}, the 3-manifold is decomposed into two Seifert manifolds. 
The base surfaces of these are annuli with one exceptional points, and the orders of them are $|p|$ and $|r|$. 
Of course, the entire 3-manifold is Seifert if $p=\pm1$ or $r=\pm1$. 
Note that neither $p$ nor $r$ is $0$ since $P_{2p+1,2q+1,2r+1}$ is not the unknot and $-1\in \{p,q,r\}$.

\begin{figure}[htbp]
 \begin{center}
  \includegraphics[width=100mm]{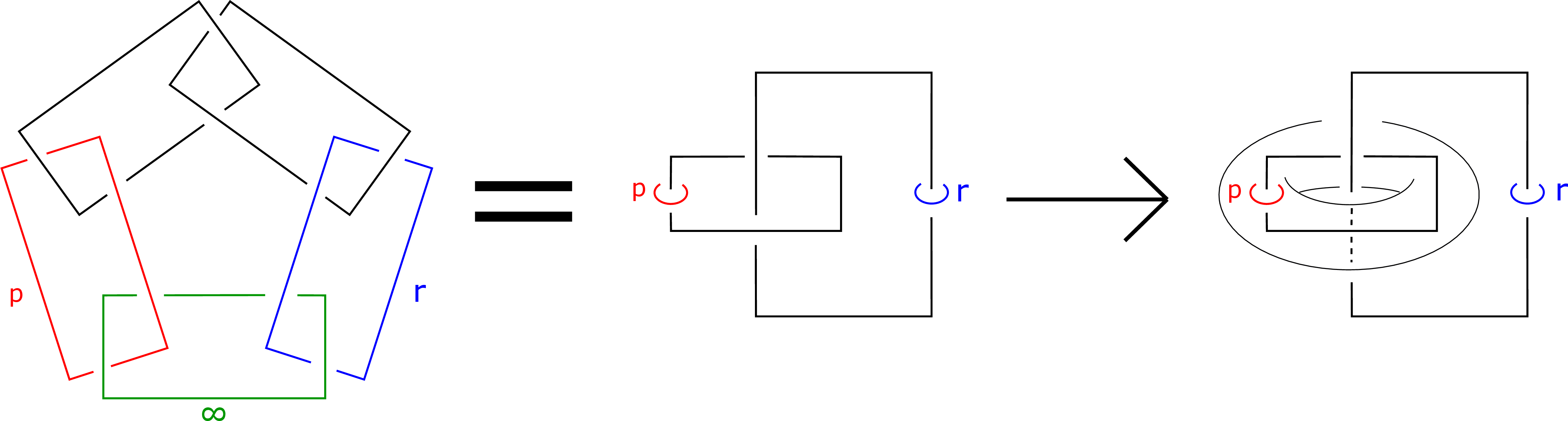}
 \end{center}
 \caption{A surgery description for $q=-1$}
 \label{5-chain_mirror_qminus1}
\end{figure}

\subsubsection{The case where $0\in \{p,q,r\}$ }
Since $P_{2p+1,2q+1,2r+1}$ is equivalent to $P_{2q+1,2r+1,2p+1}$ and the standard Seifert surface $T$ is compatible with this equivalence, we assume that $q=0$. 
We use Proposition~\ref{xpqr} and the surgery description in Figure~\ref{5chain_mirror}. 
Then we change the diagram as in Figure~\ref{5chain_mirror_q0}. 
By the torus depicted in the right of Figure~\ref{5chain_mirror_q0}, the 3-manifold is decomposed into two Seifert manifolds. 
The base surfaces of these are annuli with one exceptional points, and the orders of them are $|p+1|$ and $|r+1|$. 
Of course, the entire 3-manifold is Seifert if $p+1=\pm1$ or $r+1=\pm1$. 
Note that neither $p$ nor $r$ is $-1$ since $P_{2p+1,2q+1,2r+1}$ is not the unknot and $0\in \{p,q,r\}$.

\begin{figure}[htbp]
 \begin{center}
  \includegraphics[width=120mm]{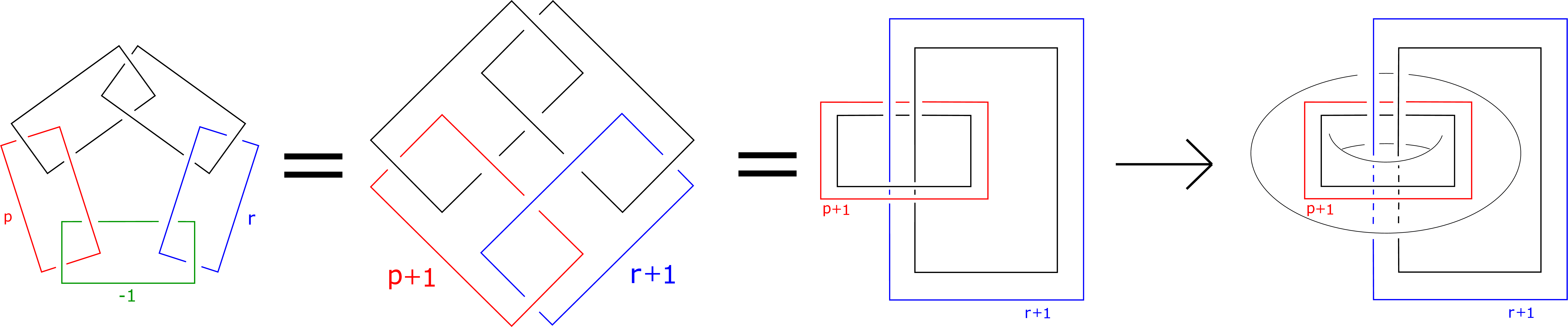}
 \end{center}
 \caption{A surgery description for $q=0$}
 \label{5chain_mirror_q0}
\end{figure}

\subsubsection{The case where $\{-2,1\}\subset \{p,q,r\}$}
Though the fact that $\overline{P_{2p+1,-3,3}(0)\setminus T'}$ is homeomorphic to the complement of $(2,4)$-torus link is stated at Theorem 1.5 of \cite{cantwell}, we give a proof for the completeness. 
Since $P_{2p+1,2q+1,2r+1}$ is equivalent to $P_{2q+1,2r+1,2p+1}$ and the standard Seifert surface $T$ is compatible with this equivalence, we assume that $q=-2$. 
Moreover, $P_{2\cdot 1+1,2(-2)+1,2r+1}$ is the mirror of $P_{2(-2)+1,2\cdot1+1,2(-r-1)+1}$ and $P_{2(-2)+1,2\cdot1+1,2(-r-1)+1}$ is equivalent to $P_{2(-r-1)+1,2(-2)+1,2\cdot1+1}$, and the standard Seifert surface $T$ is compatible with these equivalences, we assume that $r=1$. 
We use Proposition~\ref{xpqr} and the surgery description in Figure~\ref{5chain_mirror}. 
Then we change the diagram as in Figure~\ref{5chain_mirror_qminus2r1}. 
In the last of Figure~\ref{5chain_mirror_qminus2r1}, we can give the 3-manifold a Seifert structure such that the meridian of $\frac{1}{2}$-framed unknot is a regular fiber. 
Under this Seifert structure, the base surface is annulus with one exceptional point of order two. 

\begin{figure}[htbp]
 \begin{center}
  \includegraphics[width=140mm]{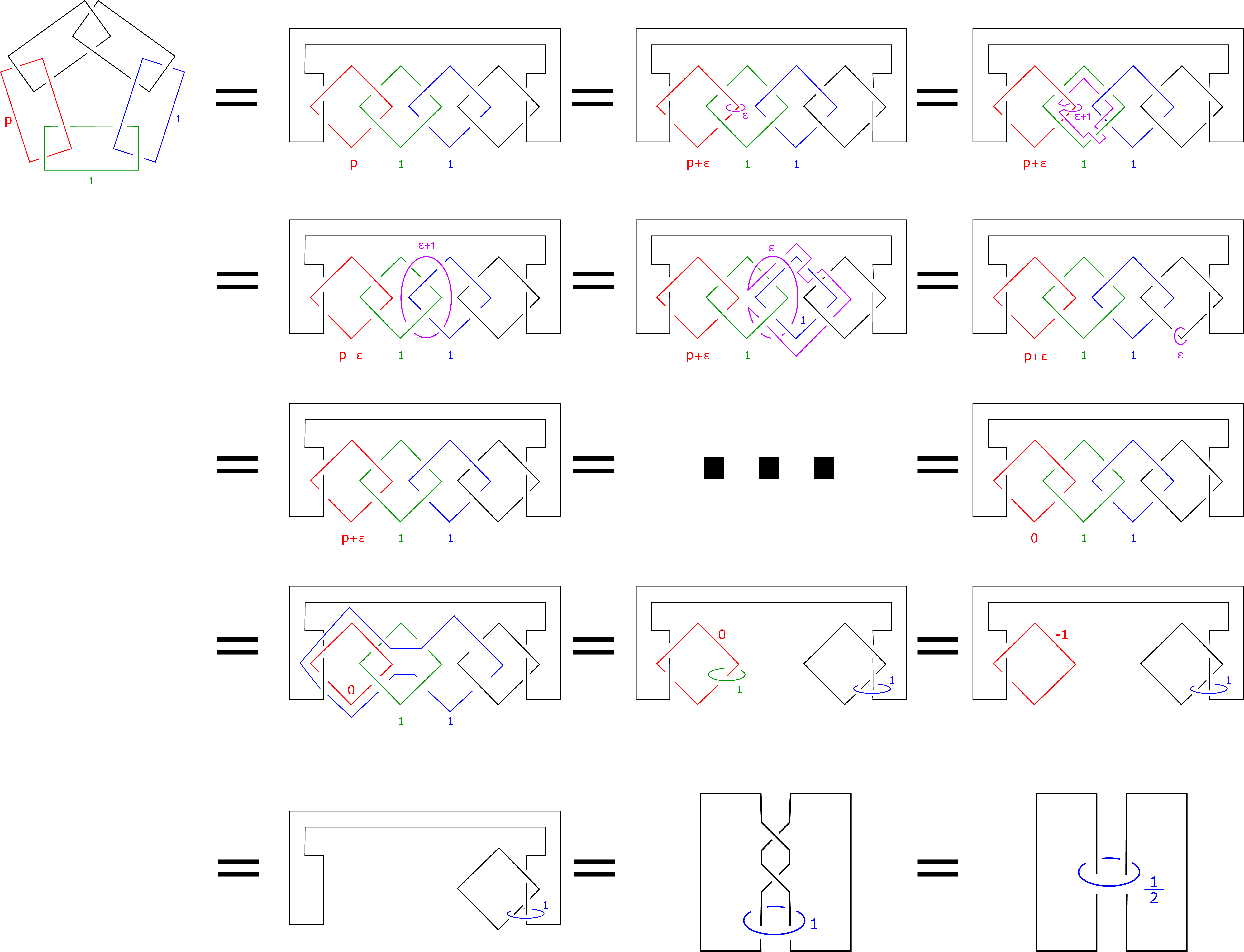}
 \end{center}
 \caption{A surgery description for $q=-2$ and $r=1$, where $\epsilon$ is $1$ or $-1$}
 \label{5chain_mirror_qminus2r1}
\end{figure}

\subsubsection{The case where $\{p,q,r\}=\{-2,2,2\}$}
Since $P_{2p+1,2q+1,2r+1}$ is equivalent to $P_{2q+1,2r+1,2p+1}$ and the standard Seifert surface $T$ is compatible with this equivalence, we assume that $p=2$, $q=-2$ and $r=2$. 
We use Proposition~\ref{xpqr} and the surgery description in Figure~\ref{5chain_mirror}. 
Then we change the diagram as in Figure~\ref{5chain_mirror_p2qminus2r2}. 
The last of Figure~\ref{5chain_mirror_p2qminus2r2} can be decomposed into two pieces, one of which is the complement of a small neighborhood of the left-handed trefoil as in the top of Figure~\ref{p2qminus2r2} by splitting along a torus. 
Note that the complement of a small neighborhood of the left-handed trefoil is a Seifert manifold since the trefoil is a fibered knot of genus one whose monodromy is periodic of order six. 
Concretely, its base surface is a disk with two exceptional points of order two and three. 
Note that a regular fiber of it is the longitude of the trefoil which is obtained by spiraling the canonical one six times along the meridian. 
The other piece is $\Sigma_{0,3}\times S^1$ as we can see in the bottom of Figure~\ref{p2qminus2r2}, where a shaded surface in the right is one of its fiber. 
Since a regular fiber of this $\Sigma_{0,3}\times S^1$-piece is a meridian of the trefoil in the other piece, the entire 3-manifold is not a Seifert manifold. 

\begin{figure}[htbp]
 \begin{center}
  \includegraphics[width=100mm]{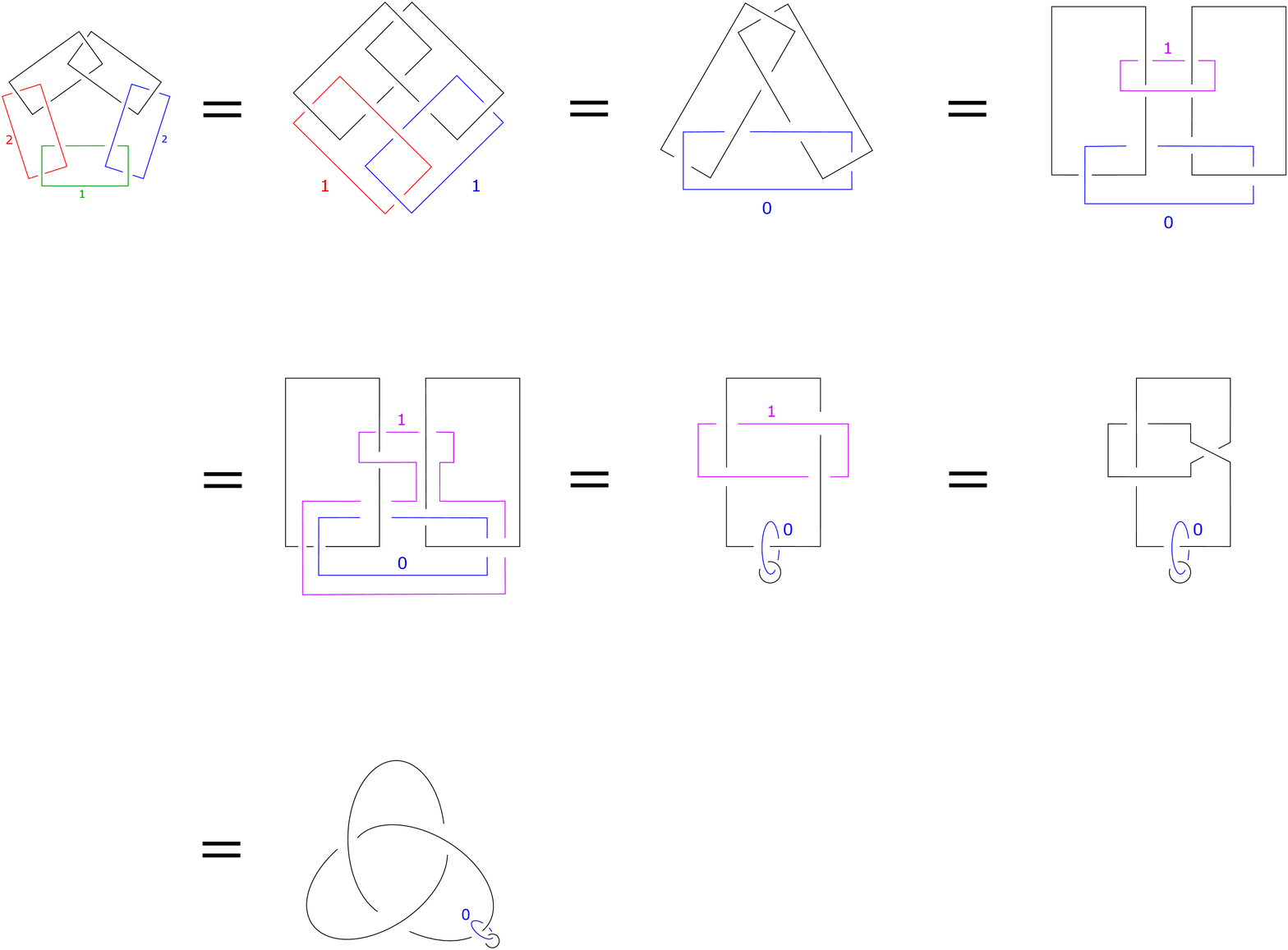}
 \end{center}
 \caption{A surgery description for $p=2$, $q=-2$ and $r=2$}
 \label{5chain_mirror_p2qminus2r2}
\end{figure}

\begin{figure}[htbp]
 \begin{center}
  \includegraphics[width=80mm]{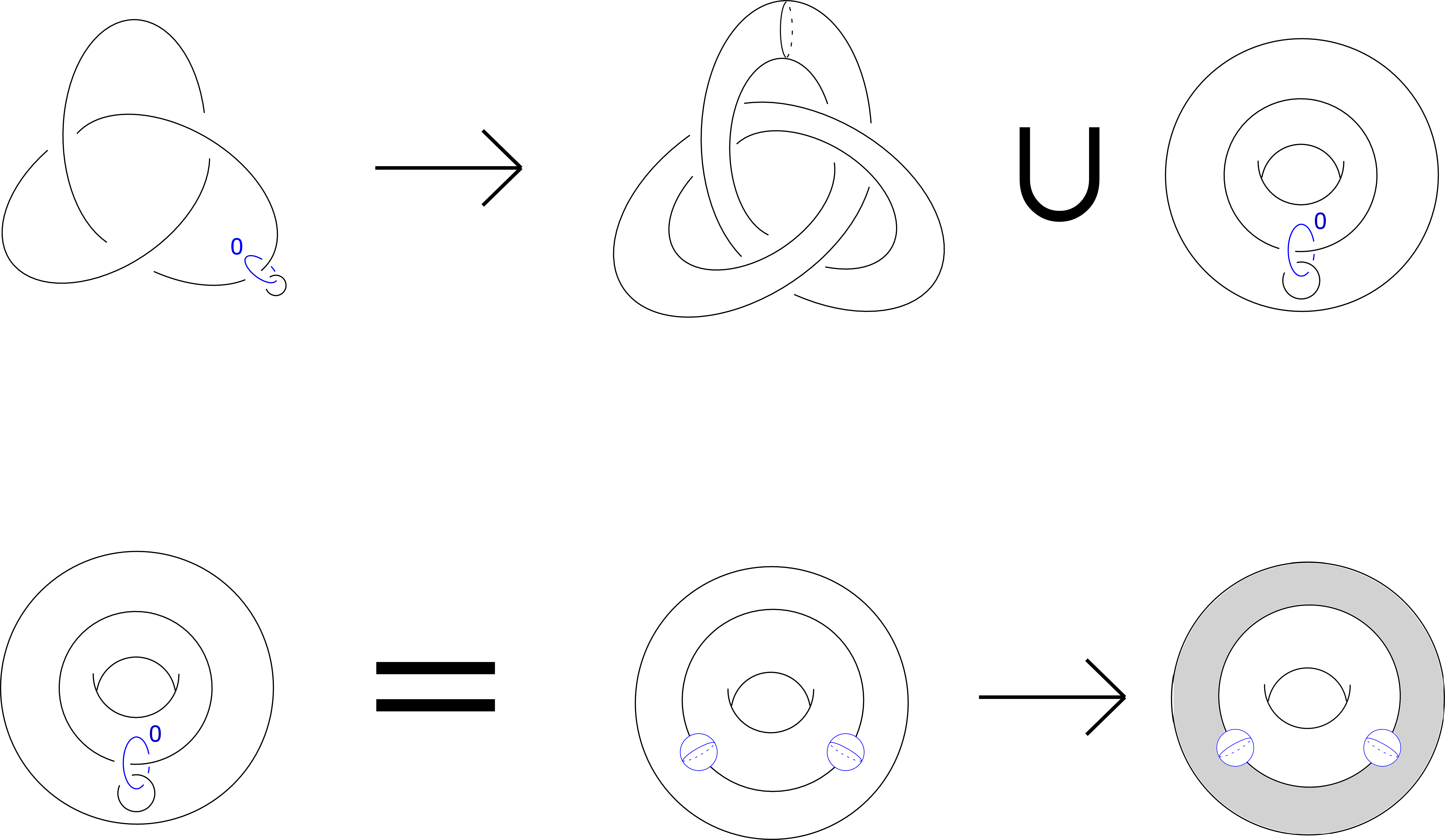}
 \end{center}
 \caption{Top: splitting along a torus\ \ \ \ \ Bottom: description for $\Sigma_{0,3}\times S^{1}$}
 \label{p2qminus2r2}
\end{figure}

\subsubsection{The case where $\{p,q,r\}=\{-3,-3,1\}$}
Since $P_{2p+1,2q+1,2r+1}$ is equivalent to $P_{2q+1,2r+1,2p+1}$ and the standard Seifert surface $T$ is compatible with this equivalence, we assume that $p=-3$, $q=1$ and $r=-3$. 
Since $P_{2(-3)+1,2\cdot1+1,2(-3)+1}$ is the mirror of $P_{2\cdot2+1,2(-2)+1,2\cdot2+1}$ and the standard Seifert surface $T$ is compatible with this equivalence, $\overline{P_{2(-3)+1,2\cdot1+1,2(-3)+1}(0)\setminus T'}$ is the mirror of $\overline{P_{2\cdot2+1,2(-2)+1,2\cdot2+1}(0)\setminus T'}$. 
Especially, there is a torus which splits the 3-manifold into the complement of a small neighborhood of the right-handed trefoil and $\Sigma_{0,3}\times S^1$.

\subsection{Hyperbolic cases}\label{hypcase}
In this section, we show that $\overline{P_{2p+1,2q+1,2r+1}(0)\setminus T'}$ is hyperbolic if $\{p,q,r\}$ is not in the cases listed in subsection~\ref{nonhyp}. 
We use the results in \cite{martelli}, which completely classify the exceptional (i.e. giving non-hyperbolic 3-manifold) Dehn fillings along minimally twisted 5-chain link, denoted by L10n113 in Thistlewaite's name and so on. 
Note that L10n113 is the mirror image of $L$ in the top of Figure~\ref{5chain_mirror_wirtinger}. 
We list the results in \cite{martelli} we need. 
As a notation, filling slopes of a component of a link are elements of $\mathbb{Q}\cup \{\infty\}\cup \{\phi\}$. 
A filling slope $\phi$ means leaving the corresponding component unfilled. 
For any self-function $f$ of $\mathbb{Q}\cup \{\infty\} $, we regard this as a self-function of $\mathbb{Q}\cup \{\infty\}\cup \{\phi\}$ by setting $f(\phi)=\phi$. 

\begin{fact}\label{fact5chain}(Theorem 0.1. of \cite{martelli})\\
Take {\rm L10n113} and give it filling slopes $(\alpha_{1}, \alpha_2,\alpha_3,\alpha_4,\alpha_5)\in \Bigl( \mathbb{Q}\cup \{\infty\}\cup \{\phi\} \Bigr)^{5}$ as in Figure~\ref{5chain}. 
Then the resultant of the filling is non-hyperbolic if and only if up to a composition of the following maps;
\begin{itemize}
\item[(F 3.1.1)] $(\alpha_{1}, \alpha_2,\alpha_3,\alpha_4,\alpha_5) \longmapsto (\alpha_{5}, \alpha_1,\alpha_2,\alpha_3,\alpha_4)$, 
\item[(F 3.1.2)] $(\alpha_{1}, \alpha_2,\alpha_3,\alpha_4,\alpha_5) \longmapsto (\alpha_{5}, \alpha_4,\alpha_3,\alpha_2,\alpha_1)$,
\item[(F 3.1.3)] $(\alpha_{1}, \alpha_2,\alpha_3,\alpha_4,\alpha_5) \longmapsto(\frac{1}{\alpha_{2}}, \frac{1}{\alpha_1},1-\alpha_3,\frac{\alpha_4}{\alpha_{4}-1},1-\alpha_5)$,
\item[(F 3.1.4)] $(-1, \alpha_2,\alpha_3,\alpha_4,\alpha_5) \longmapsto (-1, \alpha_{3}-1,\alpha_4,\alpha_5+1,\alpha_2)$,
\item[(F 3.1.5)] $(-1, -2,-2,-2,\alpha) \longmapsto (-1, -2,-2,-2,-\alpha-6)$,
\end{itemize}

 $(\alpha_{1}, \alpha_2,\alpha_3,\alpha_4,\alpha_5)$ contains one of the following.

\begin{itemize}
\item[] $\infty$,\hspace{1.0cm}$(-1,-2,-2,-1)$,\hspace{1.0cm}$(-2,-\frac{1}{2},3,3,-\frac{1}{2})$,\hspace{1.0cm}$(-1,-2,-2,-3,-5)$,
\item[]  $(-1,-2,-3,-2,-4)$,\hspace{1.0cm}$(-1,-3,-2,-2,-3)$,\hspace{1.0cm}$(-2,-2,-2,-2,-2)$
\end{itemize}
\end{fact}

\begin{figure}[htbp]
 \begin{center}
  \includegraphics[width=40mm]{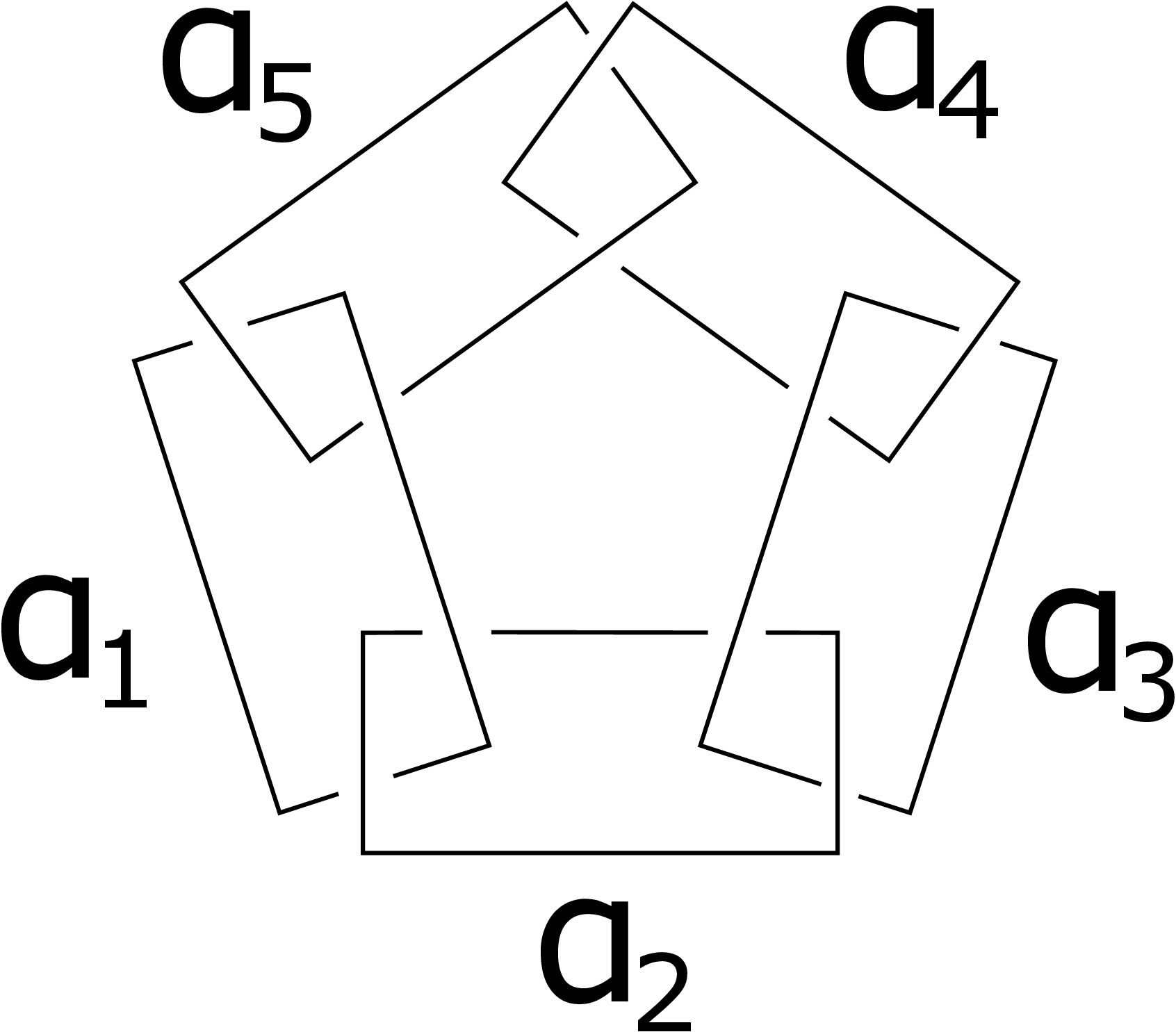}
 \end{center}
 \caption{A link L10n113 with filling slopes}
 \label{5chain}
\end{figure}

By $(-1)$-filling to any component of L10n113, we get a link L8n7 in Thistlewaite's name (or $8^{4}_{2}$ in Rolfsen's name). 
And by $(-1)$-filling to any component of L8n7, we get a link L6a5 in Thistlewaite's name (or $6^{3}_{1}$ in Rolfsen's name). 
Thus the following facts are contained in Fact~\ref{fact5chain}. 
However, we list these for our use. 

\begin{fact}\label{fact4chain} (Theorem 3.5. of \cite{martelli})\\
Take {\rm L8n7} and give it filling slopes $(\alpha_{1}, \alpha_2,\alpha_3,\alpha_4)\in \Bigl( \mathbb{Q}\cup \{\infty\}\cup \{\phi\} \Bigr)^{4}$ as in Figure~\ref{4chain}. 
Then the resultant of the filling is non-hyperbolic if and only if up to a composition of the following maps;
\begin{itemize}
\item[(F 3.2.1)] $(\alpha_{1}, \alpha_2,\alpha_3,\alpha_4) \longmapsto (\alpha_{4}, \alpha_1,\alpha_2,\alpha_3)$, 
\item[(F 3.2.2)] $(\alpha_{1}, \alpha_2,\alpha_3,\alpha_4) \longmapsto (\alpha_4,\alpha_3,\alpha_2,\alpha_1)$,
\item[(F 3.2.3)] $(\alpha_{1}, \alpha_2,\alpha_3,\alpha_4) \longmapsto(\frac{\alpha_{1}-2}{\alpha_{1}-1}, \frac{\alpha_{2}-2}{\alpha_{2}-1},\frac{\alpha_{3}-2}{\alpha_{3}-1},\frac{\alpha_{4}-2}{\alpha_{4}-1})$,
\item[(F 3.2.4)] $(\alpha_{1}, \alpha_2,\alpha_3,\alpha_4 \longmapsto (2-\alpha_1,\frac{\alpha_2}{\alpha_{2}-1}, 2-\alpha_3,\frac{\alpha_{4}}{\alpha_4-1})$,
\item[(F 3.2.5)] $(-1, \alpha_2,\alpha_3,\alpha_4) \longmapsto (-1, \alpha_{3}-1,\alpha_2+1,\alpha_4)$,
\item[(F 3.2.6)] $(-1, -2,-2,\alpha) \longmapsto (-1, -2,-2,-\alpha-4)$,
\end{itemize}

 $(\alpha_{1}, \alpha_2,\alpha_3,\alpha_4)$ contains one of the following.

\begin{itemize}
\item[] $0$,\hspace{1.0cm}$\infty$,\hspace{1.0cm}$(-1,-2,-1)$,\hspace{1.0cm}$(-2,-2,-2,-2)$,
\item[] $(-1,-3,-2,-3)$,\hspace{1.0cm}$(-1,-2,-3,-4)$
\end{itemize}
\end{fact}

\begin{figure}[htbp]
 \begin{center}
  \includegraphics[width=30mm]{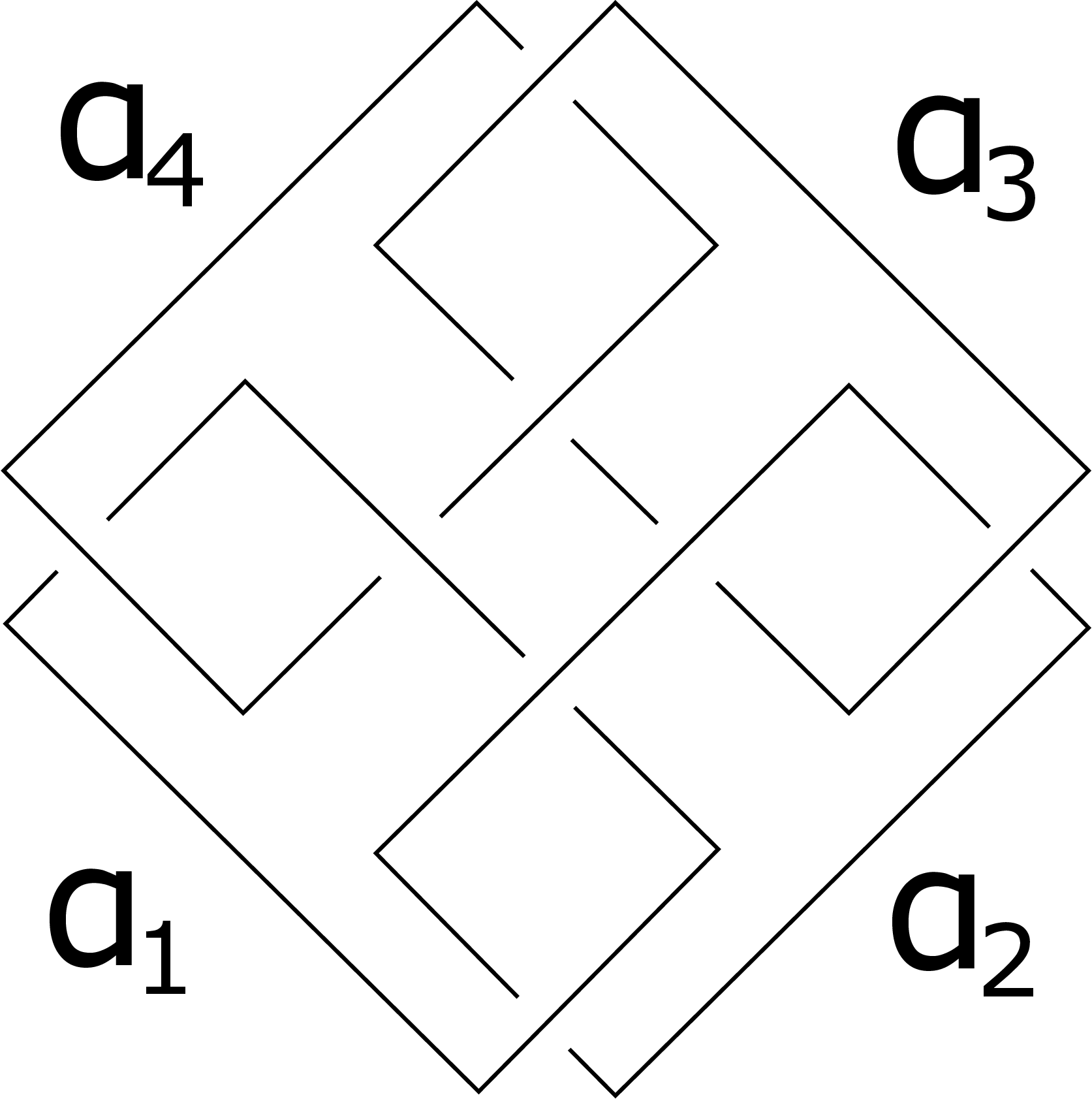}
 \end{center}
 \caption{A link L8n7 with filling slopes}
 \label{4chain}
\end{figure}

\begin{fact}\label{fact3chain} (Theorem 3.3. of \cite{martelli})\\
Take {\rm L6a5} and give it filling slopes $(\alpha_{1}, \alpha_2,\alpha_3)\in \Bigl( \mathbb{Q}\cup \{\infty\}\cup \{\phi\} \Bigr)^{3}$ as in Figure~\ref{3chain}. 
Then the resultant of the filling is non-hyperbolic if and only if up to a composition of the following maps;
\begin{itemize}
\item[(F 3.3.1)] $(\alpha_{1}, \alpha_2,\alpha_3) \longmapsto (\alpha_{3},\alpha_2,\alpha_1)$, 
\item[(F 3.3.2)] $(\alpha_{1}, \alpha_2,\alpha_3) \longmapsto (\alpha_2,\alpha_1,\alpha_3)$,
\item[(F 3.3.3)] $(\frac{1}{2}, \alpha_2,\alpha_3) \longmapsto(\frac{1}{2}, 4-\alpha_{2},4-\alpha_3)$,
\item[(F 3.3.4)] $(\frac{3}{2}, \alpha_2,\alpha_3) \longmapsto(\frac{3}{2}, \frac{2\alpha_{2}-5}{\alpha_{2}-2},\frac{2\alpha_{3}-5}{\alpha_{3}-2})$,
\item[(F 3.3.5)] $(\frac{5}{2}, \alpha_2,\alpha_3) \longmapsto(\frac{5}{2}, \frac{\alpha_{2}-3}{\alpha_{2}-2},\frac{2\alpha_{3}-3}{\alpha_{3}-1})$,
\item[(F 3.3.6)] $(4, \alpha_2,\alpha_3) \longmapsto (4, \frac{\alpha_2-2}{\alpha_2-1},\frac{\alpha_3-2}{\alpha_3-1})$,
\item[(F 3.3.7)] $(-1, -2,\alpha) \longmapsto (-1, -2,-\alpha-2)$,
\item[(F 3.3.8)] $(-1, 4,\alpha) \longmapsto (-1, 4,\frac{1}{\alpha})$,
\end{itemize}

 $(\alpha_{1}, \alpha_2,\alpha_3)$ contains one of the following.

\begin{itemize}
\item[] $0$, \ $1$, \ $2$, \ $3$, \ $\infty$, \ $(-1,-1)$, \ $(-1,-3,-3)$, \ $(-2,-2,-2)$, \ $(\frac{2}{3},4,4)$
\end{itemize}
\end{fact}

\begin{figure}[htbp]
 \begin{center}
  \includegraphics[width=30mm]{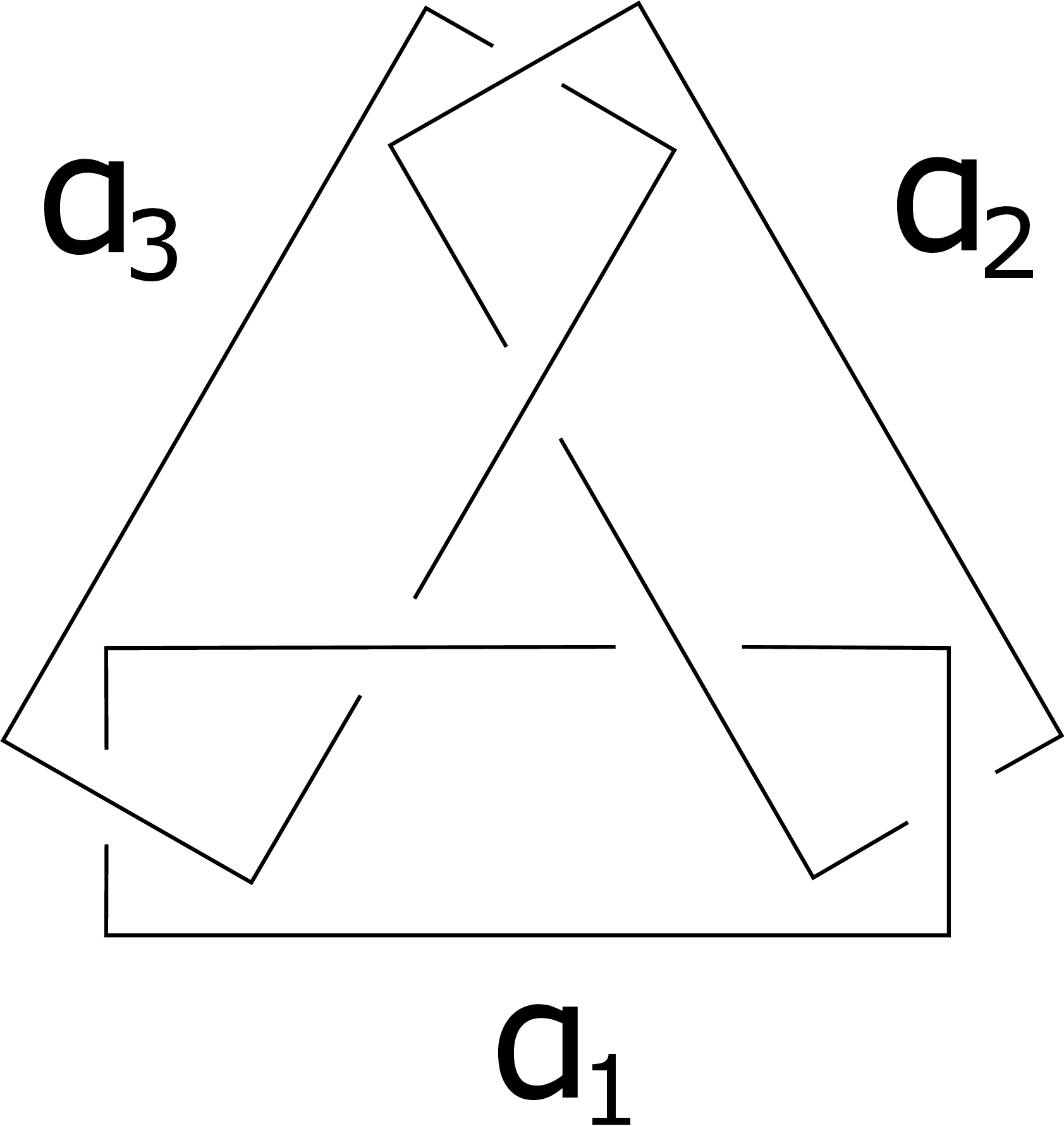}
 \end{center}
 \caption{A link L6a7 with filling slopes}
 \label{3chain}
\end{figure}

\begin{rmk}\label{onefilling}
We use Fact~\ref{fact3chain} only in one component filling i.e. filling slope $(\alpha,\phi,\phi)$ for $\alpha \in \mathbb{Q}\cup \{\infty\}$. 
Note that the position of $\alpha$ is not important because of the operations $(F 3.3.1)$ and $(F 3.3.2)$. 
In this case, the orbit of $(\alpha,\phi,\phi)$ under the operations from $(F 3.3.1)$ to $(F 3.3.8)$ is easily computed. 
As a result, we see that the filling of L6a7 with filling slope $(\alpha,\phi,\phi)$ is exceptional if and only if $\alpha=0$, $1$, $2$, $3$ or $\infty$. 
This result is already stated in \cite{martelli2}. 
Note that three-chain links in \cite{martelli2} and \cite{martelli} are mirror of each other.  
\end{rmk}

\subsubsection{The case where $-2, -1, 0, 1\notin \{p,q,r\}$}
Take a surgery description in Figure~\ref{5chain_mirror} for $\overline{P_{2p+1,2q+1,2r+1}(0)\setminus T'}$. 
By taking its mirror, we see that this is the mirror of the resultant of the filling of $L10n113$ in Figure~\ref{5chain} with filling slope $(-r,\frac{1}{q+1},-p,\phi,\phi)$. 
Thus to prove that this is hyperbolic, it is enough to prove that every element in the orbit of $(-r,\frac{1}{q+1},-p,\phi,\phi)$ under operations from $(F 3.1.1)$ to $(F 3.1.5)$ does not contain $\infty$ by Fact~\ref{fact5chain}. 
Set $f_1$, $f_2$ and $f_3$ to be functions such that $f_{1}(x)=\frac{1}{x}$, $f_{2}(x)=1-x$ and $f_{3}(x)=\frac{x}{x-1}$, which are functions appearing in the entries in operations from $(F 3.1.1)$ to $(F 3.1.3)$. 
Note that the orbit of $-1$ under $f_{1}$, $f_{2}$ and $f_{3}$ is $\{-1,\frac{1}{2},2\}$. 
Then under the condition $-2, -1, 0, 1\notin \{p,q,r\}$, the operations $(F 3.1.4)$ and $(F 3.1.5)$ do not occur in the computation of the orbit of $(-r,\frac{1}{q+1},-p,\phi,\phi)$ under operations from $(F 3.1.1)$ to $(F 3.1.5)$. 
Moreover note that the orbit of $\infty$ under $f_{1}$, $f_{2}$ and $f_{3}$ is $\{0,1,\infty \}$.  
Thus under the condition $-2, -1, 0, 1\notin \{p,q,r\}$, every element in the orbit of $(-r,\frac{1}{q+1},-p,\phi,\phi)$ under operations from $(F 3.1.1)$ to $(F 3.1.3)$ does not contain $\infty$. 
Therefore we can conclude that $\overline{P_{2p+1,2q+1,2r+1}(0)\setminus T'}$ is hyperbolic if $-2, -1, 0, 1\notin \{p,q,r\}$.

\subsubsection{The case where $1\in \{p,q,r\}$}
Since $P_{2p+1,2q+1,2r+1}$ is equivalent to $P_{2q+1,2r+1,2p+1}$ and the standard Seifert surface $T$ is compatible with this equivalence, we assume that $r=1$. 

\subparagraph{3.2.2.1 The case where $-3, -2, -1, 0, 1 \notin\{p,q\}$} 

Take a surgery description in Figure~\ref{5chain_mirror} for $\overline{P_{2p+1,2q+1,2\cdot1+1}(0)\setminus T'}$. 
In this case, we can change the diagram to a four component link as in Figure~\ref{hyp_r1}. 
By taking its mirror, we see that this is the mirror of the resultant of the filling of $L8n7$ in Figure~\ref{4chain} with filling slope $(\frac{q+2}{q+1},-p,\phi,\phi)$. 
Thus to prove that this is hyperbolic, it is enough to prove that every element in the orbit of $(\frac{q+2}{q+1},-p,\phi,\phi)$ under operations from $(F 3.2.1)$ to $(F 3.2.6)$ does not contain $0$ or $\infty$ by Fact~\ref{fact4chain}. 
Set $f_1$, $f_2$ and $f_3$ to be functions such that $f_{1}(x)=\frac{x-2}{x-1}$, $f_{2}(x)=2-x$ and $f_{3}(x)=\frac{x}{x-1}$, which are functions appearing in the entries in operations from $(F 3.2.1)$ to $(F 3.2.4)$. 
Note that the orbit of $-1$ under $f_{1}$, $f_{2}$ and $f_{3}$ is $\{-1,\frac{1}{2},\frac{3}{2},3\}$. 
Then under the condition $-3, -2, -1, 0, 1\notin \{p,q\}$, the operations $(F 3.2.5)$ and $(F 3.2.6)$ do not occur in the computation of the orbit of $(\frac{q+2}{q+1},-p,\phi,\phi)$ under operations from $(F 3.2.1)$ to $(F 3.2.6)$. 
Moreover note that  the orbit of $0$ under $f_{1}$, $f_{2}$ and $f_{3}$ is $\{0,2 \}$ and that of $\infty$ under $f_{1}$, $f_{2}$ and $f_{3}$ is $\{1,\infty \}$.  
Thus under the condition $-3, -2, -1, 0, 1\notin \{p,q\}$, every element in the orbit of $(\frac{q+2}{q+1},-p,\phi,\phi)$ under operations from $(F 3.2.1)$ to $(F 3.2.4)$ contains neither $0$ nor $\infty$. 
Therefore we can conclude that $\overline{P_{2p+1,2q+1,2\cdot 1+1}(0)\setminus T'}$ is hyperbolic if $-3, -2, -1, 0, 1\notin \{p,q\}$.

\begin{figure}[htbp]
 \begin{center}
  \includegraphics[width=70mm]{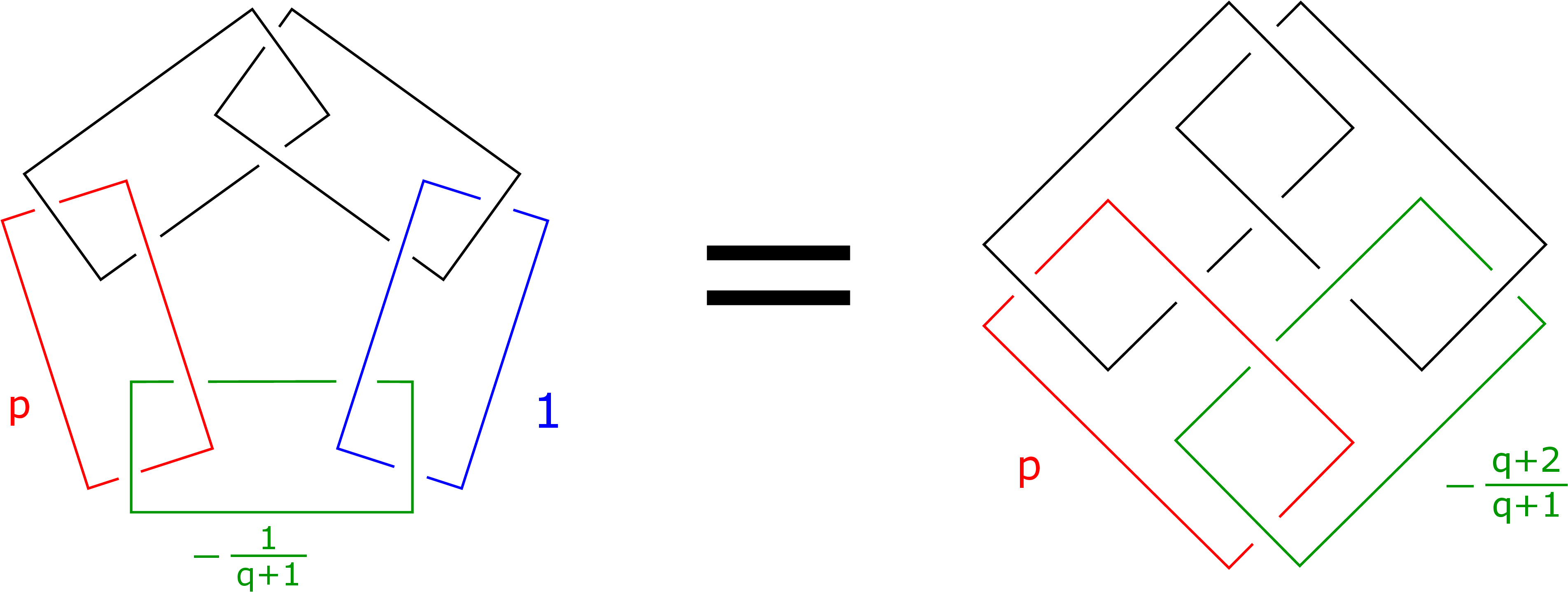}
 \end{center}
 \caption{A surgery description of $\overline{P_{2p+1,2q+1,2\cdot1+1}(0)\setminus T'}$}
 \label{hyp_r1}
\end{figure}

\subparagraph{3.2.2.2 The case where $1\in \{p,q\}$}

Since $P_{2p+1,2\cdot1+1,2\cdot1+1}$ is equivalent to $P_{2\cdot1+1,2p+1,2\cdot1+1}$ and the standard Seifert surface $T$ is compatible with this equivalence, we assume that $p=1$. 
Take a surgery description in Figure~\ref{5chain_mirror} for $\overline{P_{2\cdot1+1,2q+1,2\cdot1+1}(0)\setminus T'}$. 
In this case, we can change the diagram to a three component link as in Figure~\ref{hyp_p1r1}. 
By taking its mirror, we see that this is the mirror of the resultant of the filling of L6a7 in Figure~\ref{3chain} with filling slope $(\frac{2q+3}{q+1},\phi,\phi)$. 
By Remark~\ref{onefilling}, we see that $\overline{P_{2\cdot1+1,2q+1,2\cdot1+1}(0)\setminus T'}$ is hyperbolic if and only if $q\neq -2, -1, 0$. 
The other cases, where $q=-2,-1,0$ (under $p=r=1$) have already contained in the cases in subsection~\ref{nonhyp}. 

\begin{figure}[htbp]
 \begin{center}
  \includegraphics[width=100mm]{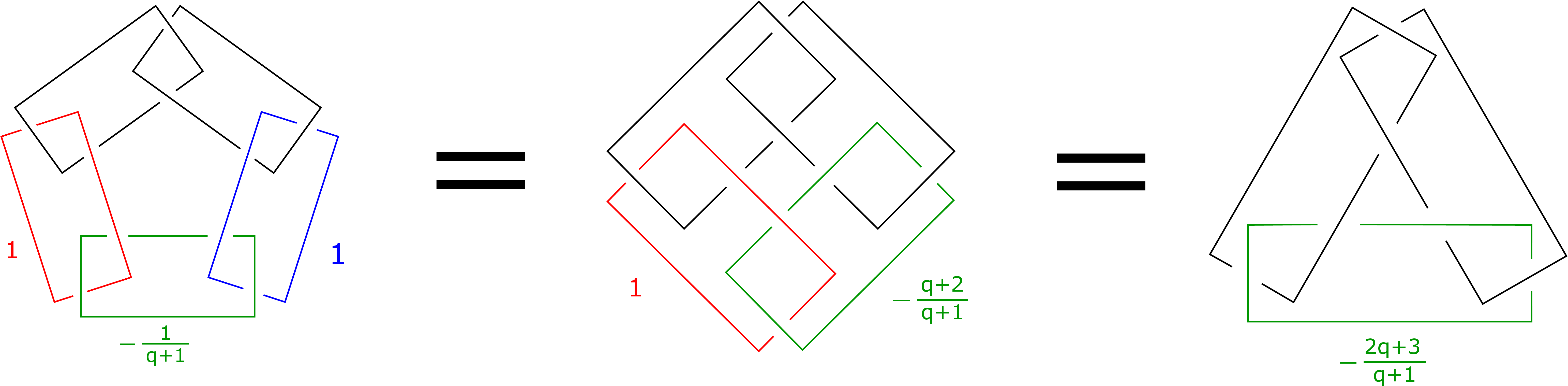}
 \end{center}
 \caption{A surgery description of $\overline{P_{2\cdot1+1,2q+1,2\cdot1+1}(0)\setminus T'}$}
 \label{hyp_p1r1}
\end{figure}

\subparagraph{3.2.2.3 The case where $-3\in \{p,q\}$}
If $q=-3$, take a surgery description in Figure~\ref{5chain_mirror} for $\overline{P_{2p+1,2(-3)+1,2\cdot1+1}(0)\setminus T'}$ and
in this case, we can change the diagram to a three component link as in Figure~\ref{hyp_qminus3r1}. 
Note that this is the resultant of the filling of L6a7 in Figure~\ref{3chain} with filling slope $(p+3,\phi,\phi)$. 
By Remark~\ref{onefilling}, we see that $\overline{P_{2p+1,2(-3)+1,2\cdot1+1}(0)\setminus T'}$ is hyperbolic if and only if $p\neq -3, -2, -1, 0$. 
The other cases, where $p=-3 ,-2,-1,0$ (under $q=-3$ and $r=1$) have already contained in the cases in subsection~\ref{nonhyp}. \\
If $p=-3$, we will change variables. 
Since $P_{2(-3)+1,2q+1,2\cdot1+1}$ is equivalent to $P_{2\cdot1+1,2(-3)+1,2q+1}$ and the standard Seifert surface $T$ is compatible with this equivalence, we consider $\overline{P_{2\cdot+1,2(-3)+1,2q+1}(0)\setminus T'}$ instead of $\overline{P_{2(-3)+1,2q+1,2\cdot1+1}(0)\setminus T'}$. 
Take a surgery description in Figure~\ref{5chain_mirror} for $\overline{P_{2\cdot1+1,2(-3)+1,2q+1}(0)\setminus T'}$. 
In this case, we can change the diagram to a three component link as in Figure~\ref{hyp_p1qminus3}. 
Note that the last diagram in Figure~\ref{hyp_p1qminus3} is the resultant of the filling of L6a7 in Figure~\ref{3chain} with filling slope $(\phi,\phi,q+3)$. 
By Remark~\ref{onefilling}, we see that $\overline{P_{2\cdot1+1,2(-3)+1,2q+1}(0)\setminus T'}$ is hyperbolic if and only if $q\neq -3, -2, -1, 0$. 
The other cases, where $r=-3 ,-2,-1,0$ (under $p=-3$ and $r=1$) have already contained in the cases in subsection~\ref{nonhyp}. 

\vspace{0.5cm}

As a result, we can conclude as follows. 
Suppose that $1\in \{p,q,r\}$. 
Then $\overline{P_{2p+1,2q+1,2r+1}(0)\setminus T'}$ is non-hyperbolic if and only if ``$-1\in \{p,q,r\}$'', ``$0\in \{p,q,r\}$'', ``$\{-2,1\}\subset \{p,q,r\}$'' or ``$\{p,q,r\}=\{-3,-3,1\}$''.

\begin{figure}[htbp]
 \begin{center}
  \includegraphics[width=150mm]{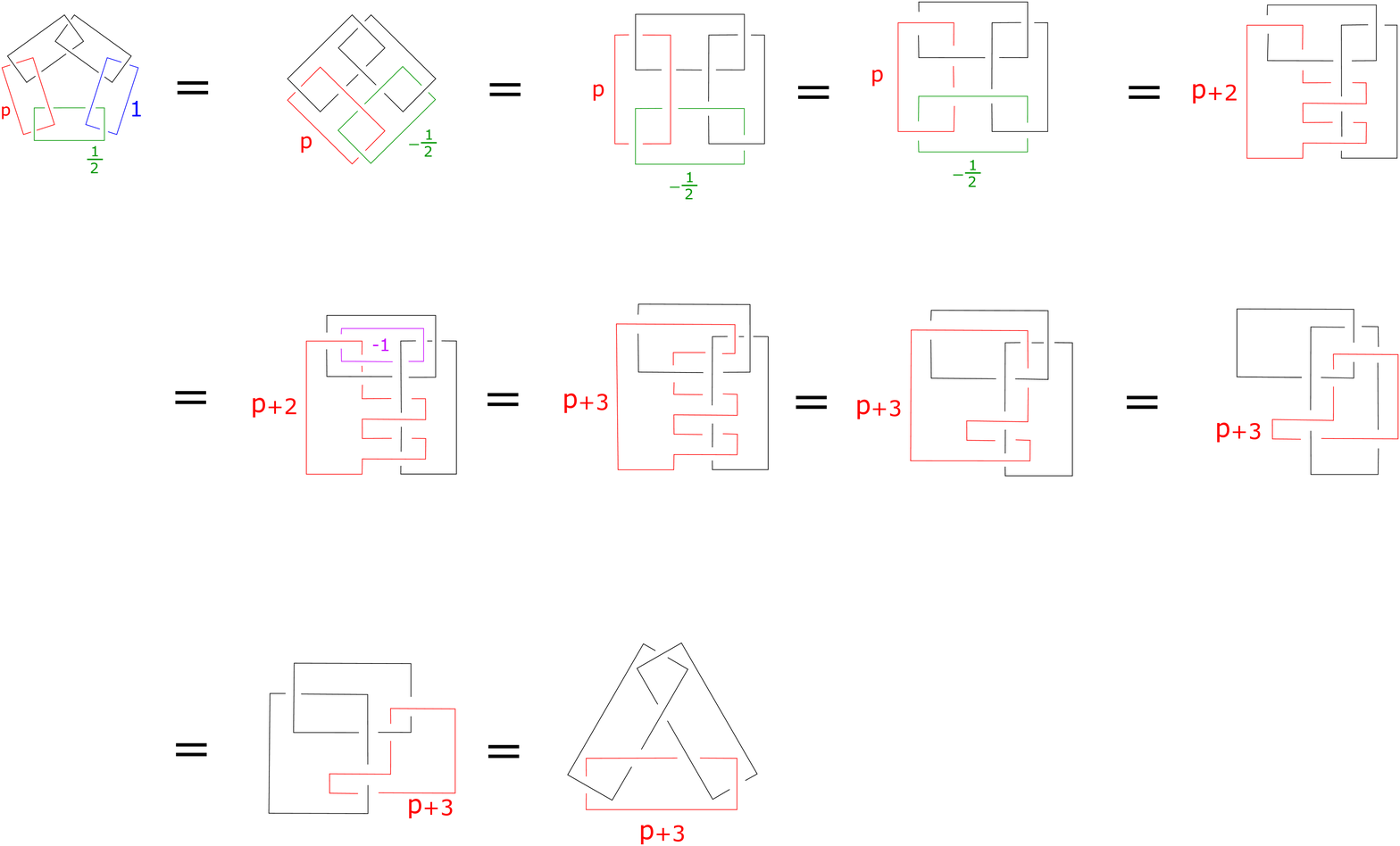}
 \end{center}
 \caption{A surgery description of $\overline{P_{2p+1,2(-3)+1,2\cdot1+1}(0)\setminus T'}$}
 \label{hyp_qminus3r1}
\end{figure}

\begin{figure}[htbp]
 \begin{center}
  \includegraphics[width=150mm]{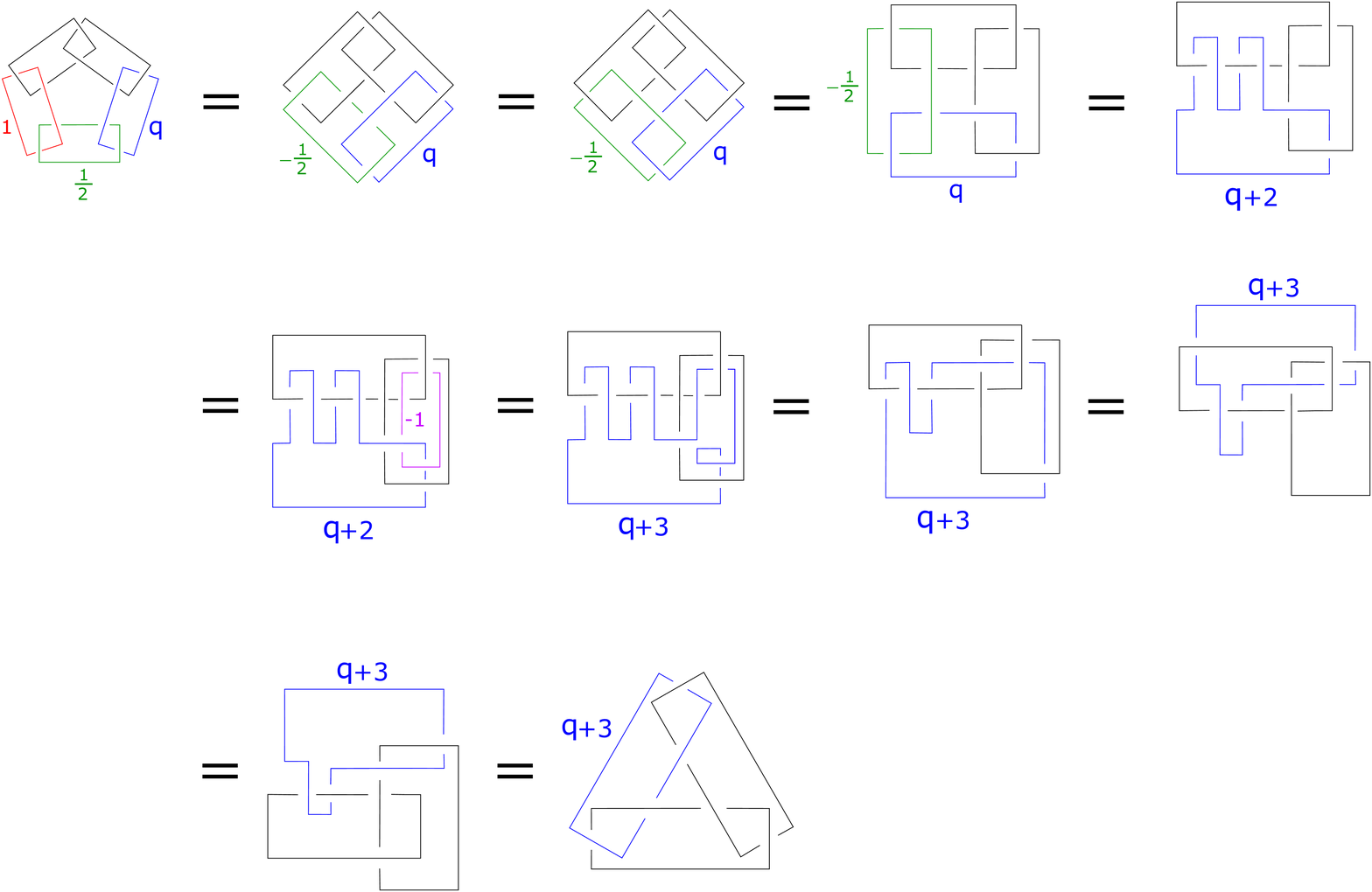}
 \end{center}
 \caption{A surgery description of $\overline{P_{2\cdot1+1,2(-3)+1,2q+1}(0)\setminus T'}$}
 \label{hyp_p1qminus3}
\end{figure}

\subsubsection{The case where $-2\in \{p,q,r\}$}
Since $P_{2p+1,2q+1,2r+1}$ is equivalent to $P_{2q+1,2r+1,2p+1}$ and the standard Seifert surface $T$ is compatible with this equivalence, we assume that $p=-2$. 
Recall that $P_{2(-2)+1,2q+1,2r+1}$ is the mirror of $P_{2\cdot1+1,2(-q-1)+1,2(-r-1)+1}$ and that the standard Seifert surface $T$ is compatible with this equivalence. 
Thus $\overline{P_{2(-2)+1,2q+1,2r+1}(0)\setminus T'}$ is orientation reversely homeomorphic to $\overline{P_{2\cdot1+1,2(-q-1)+1,2(-r-1)+1}(0)\setminus T'}$. 
By the argument for the case where $1\in \{p,q,r\}$, we see that $\overline{P_{2(-2)+1,2q+1,2r+1}\setminus T'}$ is non-hyperbolic if and only if ``$-1\in \{-q-1,-r-1\}$'', ``$0\in \{-q-1,-r-1\}$'', ``$-2 \in \{-q-1,-r-1\}$'' or ``$-q-1=-r-1=-3$''. 
Thus we have that under the assumption $-2\in \{p,q,r\}$, the 3-manifold $\overline{P_{2p+1,2q+1,2r+1}\setminus T'}$ is non-hyperbolic if and only if ``$-1\in\{p,q,r\}$'', ``$0\in\{p,q,r\}$'', ``$\{-2,1\}\subset \{p,q,r\}$'' or ``$\{p,q,r\}=\{-2,2,2\}$''. 

\vspace{0.5cm}
These arguments above finishes a proof of Proposition~\ref{proposition}.

\section{A proof of Theorem~\ref{main}} \label{secproof}
We recall the definition of JSJ-decompositions of 3-manifolds. 
\begin{defini} \label{jsj}(JSJ-decomposition \cite{jacoshalen} \cite{johannson})\\
Let $M$ be a prime compact 3-manifold with (possibly empty) incompressible boundary. 
There exists a set $\mathcal{T}$ consisting of pairwise non-parallel disjoint essential tori in $M$ such that
\begin{itemize}
\item each component obtained by cutting along the tori in $\mathcal{T}$ is either a Seifert manifold or an atoroidal manifold, and
\item there are no proper subset of $\mathcal{T}$ satisfying the above condition.
\end{itemize}
This decomposition is called the {\it JSJ-decomposition} of $M$. 
It is known that the set $\mathcal{T}$ is unique up to isotopy for a given $M$. 
By the Geometrization conjecture, an atoroidal piece of JSJ-decomposition admits a complete hyperbolic structure of finite volume. 
\end{defini}

In this section, we prove Theorem~\ref{main}, which determines the order of the set $\mathcal{T}$ of decomposing tori of the JSJ-decomposition of the 3-manifold $P_{2p+1,2q+1,2r+1}(0)$ obtained by $0$-surgery along a classical pretzel knot $P_{2p+1,2q+1,2r+1}$. 

At first, we must exclude the case where $P_{2p+1,2q+1,2r+1}$ is the unknot since the standard Seifert surface is not minimal genus in this case. 
This is done at Lemma~\ref{unknot}. 
It is well-known that the 3-manifold obtained by $0$-surgery along the unknot is $S^{2}\times S^1$, which admits a Seifert fibered structure. 
Then we have the following, which is the case $(1.1.1)$ of Theorem~\ref{main}: 
If $\{-1, 0\} \subset \{p,q,r\}$, then $\mathcal{T}=\phi$ and the entire manifold is $S^{2}\times S^1$. 
In the following, we assume that $P_{2p+1,2q+1,2r+1}$ is not the unknot. 

Next, we handle the special case where $P_{2p+1,2q+1,2r+1}$ is the trefoil. 
As is stated in \cite{ichihara}, the $0$-surgery along the trefoil is the only surgery producing a toroidal Seifert manifold among Dehn surgeries along Montesinos knots. 
Since the trefoil is a fibered knot of genus one, $\overline{P_{2p+1,2q+1,2r+1}(0)\setminus T'}$ is the product of torus and the interval if $P_{2p+1,2q+1,2r+1}$ is the trefoil. 
By using the descriptions in subsection~\ref{nonhyp}, we have $\{p,q,r\}=\{-1,\epsilon,\eta\}$ or $\{0,\epsilon-1,\eta-1\}$ for $\epsilon, \eta \in \{\pm1 \}$ if $P_{2p+1,2q+1,2r+1}$ is the trefoil. 
By drawing diagrams, we have $\{p,q,r\}=\{-1,\epsilon,\epsilon \}$ or $\{0,\epsilon-1,\epsilon-1\}$ for $\epsilon \in \{\pm 1\}$ if $P_{2p+1,2q+1,2r+1}$ is the trefoil. In the other cases, we get the figure eight knots. 
Since the trefoil is a fibered knot of genus one whose monodromy is periodic of order six, the 3-manifold obtained by $0$-surgery along the trefoil  is a closed torus bundle over $S^1$ whose monodromy is periodic of order six. 
Then we have the following, which is the case $(1.1.2)$ of Theorem~\ref{main}: 
If $\{p,q,r\}=\{-1,\epsilon,\epsilon\}$ or $\{0,\epsilon-1,\epsilon-1\}$ for $\epsilon \in \{\pm1\}$, then $\mathcal{T}=\phi$ and the entire manifold is a closed torus bundle over $S^1$ whose monodromy is periodic of order six, 
In the following, we assume that $P_{2p+1,2q+1,2r+1}$ is not the trefoil. 

Note that  the torus $T'$ obtained by capping off the standard Seifert surface $T$ is essential in $P_{2p+1,2q+1,2r+1}(0)$ since $P_{2p+1,2q+1,2r+1}$ is not the unknot, and that  
$P_{2p+1,2q+1,2r+1}(0)$ is not Seifert manifold since $P_{2p+1,2q+1,2r+1}$ is not the trefoil. Thus the set $\mathcal{T}$ of decomposing tori of the JSJ-decomposition of $P_{2p+1,2q+1,2r+1}(0)$ must contain $T'$. 
Cut $P_{2p+1,2q+1,2r+1}(0)$ along this $T'$ and continue to consider the JSJ-decomposition of $\overline{P_{2p+1,2q+1,2r+1}(0)\setminus T'}$. 
This has done in Proposition~\ref{proposition}.

\paragraph{The case (1.1.3) of Theorem~\ref{main}}
First, suppose that $\{p,q,r\}=\{-1,\epsilon,k\}$ for $\epsilon \in \{\pm1\}$ and some integer $k$ and not in the cases $(1.1.1)$ and $(1.1.2)$ of Theorem~\ref{main}. 
Then $\overline{P_{2p+1,2q+1,2r+1}(0)\setminus T'}$ is as in 3.1.1 of section 3. 
Therefore, we have that $|\mathcal{T}|=1$ and the decomposed piece is a Seifert manifold whose base surface is an annulus with one exceptional point of order $|k|$. 
Next, suppose that $\{p,q,r\}=\{0,\epsilon-1,k-1\}$ for $\epsilon \in \{\pm1\}$ and some integer $k$ and not in the cases $(1.1.1)$ and $(1.1.2)$ of Theorem~\ref{main}. 
Then $\overline{P_{2p+1,2q+1,2r+1}(0)\setminus T'}$ is as in 3.1.2 of section 3. 
Therefore, we have that $|\mathcal{T}|=1$ and the decomposed piece is a Seifert manifold whose base surface is an annulus with one exceptional point of order $|k|$. 

\paragraph{The case (1.1.4) of Theorem~\ref{main}}
First, suppose that $\{p,q,r\}=\{-1,m,n\}$ for some integers $m$ and $n$ and not in the cases $(1.1.1)$, $(1.1.2)$ and $(1.1.3)$ of Theorem~\ref{main}. 
Then $\overline{P_{2p+1,2q+1,2r+1}(0)\setminus T'}$ is as in 3.1.1 of section 3. 
Therefore, we have that $|\mathcal{T}|=2$ and the decomposed pieces are two Seifert manifolds each of whose base surfaces are annuli with one exceptional points of order $|m|$ and $|n|$. 
Next, suppose that $\{p,q,r\}=\{0,m-1,n-1\}$ for some integers $m$ and $n$ and not in the cases $(1.1.1)$, $(1.1.2)$ and $(1.1.3)$ of Theorem~\ref{main}. 
Then $\overline{P_{2p+1,2q+1,2r+1}(0)\setminus T'}$ is as in 3.1.2 of section 3. 
Therefore, we have that $|\mathcal{T}|=2$ and the decomposed pieces are two Seifert manifolds each of whose base surfaces are annuli with one exceptional points of order $|m|$ and $|n|$.

\paragraph{The case (1.1.5) of Theorem~\ref{main}}
Suppose that $\{-2,1\} \subset \{p,q,r\}$. 
Then $\overline{P_{2p+1,2q+1,2r+1}(0)\setminus T'}$ is as in 3.1.3 of section 3. 
Therefore, we have that $|\mathcal{T}|=1$ and the decomposed piece is a Seifert manifold whose base surface is an annulus with one exceptional point of order $2$. 

\paragraph{The case (1.1.6) of Theorem~\ref{main}}
First, suppose that $\{p,q,r\}=\{-2,2,2\}$. 
Then $\overline{P_{2p+1,2q+1,2r+1}(0)\setminus T'}$ is as in 3.1.4 of section 3. 
Therefore, we have that $|\mathcal{T}|=2$ and the decomposed pieces are two Seifert manifolds, one of which is $\Sigma_{0,3}\times S^1$ and the other is the complement of the trefoil. 
Next, suppose that $\{p,q,r\}=\{-3,-3,1\}$. 
Then $\overline{P_{2p+1,2q+1,2r+1}(0)\setminus T'}$ is as in 3.1.5 of section 3. 
Therefore, we have that $|\mathcal{T}|=2$ and the decomposed pieces are two Seifert manifolds, one of which is $\Sigma_{0,3}\times S^1$ and the other is the complement of the trefoil. 

\paragraph{The case (1.1.7) of Theorem~\ref{main}}
Suppose that $\{p,q,r\}$ is in none of the cases from $(1.1.1)$ to $(1.1.6)$ of Theorem~\ref{main}. 
Then by Proposition~\ref{proposition}, we have that $\overline{P_{2p+1,2q+1,2r+1}(0)\setminus T'}$ is hyperbolic. 
Therefore, we have that $|\mathcal{T}|=1$ and the decomposed piece is a hyperbolic manifold. 

\vspace{0.5cm}
These arguments above finishes a proof of Theorem~\ref{main}.


\begin{thebibliography}{50}
\bibitem{abe} T. Abe, I. D. Jong, Y. Omae and M. Takeuchi. Annulus twist and diffeomorphic 4-manifolds. Math. Proc. Cambridge Philos. Soc., 155(2):219--235, 2013. 1

\bibitem{aschenbrenner} M. Aschenbrenner, S. Friedl and H. Wilton. 3-Manifold Groups. EMS Series of Lectures in Mathematics. 2015

\bibitem{baldwin} J. Baldwin and S. Sivek. Characterizing slopes for $5_{2}$. arXiv:2209.09805.

\bibitem{baldwin2} J. Baldwin and S. Sivek. Zero-surgery characterizes infinitely many knots. arXiv:2211.04280.

\bibitem{brakes} W. R. Brakes. Manifolds with multiple knot-surgery descriptions. Math. Proc. Cambridge Philos. Soc., 87(3):443--448, 1980. 1

\bibitem{cantwell} J. Cantwell and L. Conlon. Foliations of $E(5_2)$ and related knot complements. Proceedings of the American Mathematical Society. vol.118, Num.3, July 1993.

\bibitem{gabai} D. Gabai. Foliations and the topology of 3-manifolds. III. J. Differential Geom. 26 (1987), 479--536.

\bibitem{hirasawa} M. Hirasawa and K. Murasugi. Genera and fibredness of Montesinos knots. Pacific Journal of Mathematics. Vol. 225, No. 1, 2006

\bibitem{ichihara} K. Ichihara and I. D. Jong. Toroidal Seifert fibered surgeries on Montesinos knots. Communications in Analysis and Geometry. vol.18, Num.3, 579--600, 2010

\bibitem{jacoshalen} W. Jaco and P. Shalen, Seifert fibered spaces in 3-manifolds, Memorirs A. M. S. 21 (1979)

\bibitem{johannson} K. Johannson, Homotopy equivalence of 3-manifolds with boundary, Lecture notes in Mathematics 761, Springer-Verlag, 1979

\bibitem{lackenby} M. Lackenby. Every knot has characterising slopes. Math. Ann. 374 (2019), nos. 1--2, 429--446.

\bibitem{manolescu} C. Manolescu and L. Piccirillo. From zero surgeries to candidates for exotic definite four-manifolds. arXiv:2102.04391

\bibitem{martelli2} B. Martelli and C. Piccirillo. Dehn filling of the ``magic'' 3-manifold. Comm. Anal. Geom. 14 (2006), 969--1026.

\bibitem{martelli} B. Martelli, C. Petronio and F. Roukema. Exceptional Dehn surgery on the minimally twisted five-chain link. Communications in analysis and geometry. vol.22, Number 4, 689--735, 2014.

\bibitem{miller} A. N. Miller and L. Piccirillo. Knot traces and concordance. J. Topol., 11(1):201--220, 2018. 2

\bibitem{osoinach} J. K. Osoinach, Jr. Manifolds obtained by surgery on an infinite number of knots in $S^3$. Topology, 45(4):725--733, 2006. 1

\bibitem{ozsvath} P. Ozsv\'{a}th and Z. Szab\'{o}. The Dehn surgery characterization of the trefoil and the figure eight knot. J. Symplectic Geom. 17 (2019), no. 1, 251--265.

\bibitem{piccirillo} L. Piccirillo. Shake genus and slice genus. Geom. Topol., 23(5):2665--2684, 2019. 2

\bibitem{piccirillo2} L. Piccirillo. The Conway knot is not slice. Ann. of Math. (2), 191(2):581--591, 2020. 2

\bibitem{scharlemann} M. Scharlemann. Sutured manifolds and generalized Thurston norms. J. Differential Geom. 29 (1989) 557--614.

\bibitem{sekino} N. Sekino. Generalized torsion elements in the fundamental groups of 3-manifolds obtained by $0$-surgeries along some double twist knots. Topology and its Applications. vol.343 (2024)

\bibitem{waldhausen} F. Waldhausen. On irreducible 3-manifolds which are sufficiently large. Ann. of Math. vol.87 (1968), 56--88 

\bibitem{yasui} K. Yasui. Corks, exotic 4-manifolds and knot concordance. arXiv:1505.02551


\end{thebibliography}
\end{document}